\documentclass [11pt]{amsart}
\usepackage {amsmath, amssymb, amscd, mathrsfs, hyperref, enumerate, url, graphicx, color}
\usepackage[all, cmtip]{xy}
\usepackage[text={6.5in,9in},centering,letterpaper,dvips]{geometry}

\newtheorem {theorem}{Theorem}[section]
\newtheorem {lemma}[theorem]{Lemma}
\newtheorem {proposition}[theorem]{Proposition}
\newtheorem {corollary}[theorem]{Corollary}
\newtheorem {conjecture}[theorem]{Conjecture}
\newtheorem {definition}[theorem]{Definition}
\newtheorem {question}[theorem]{Question}
\newtheorem {fact}[theorem]{Fact}
\theoremstyle{remark}
\newtheorem {remark}[theorem]{Remark}
\newtheorem {example}[theorem]{Example}

\def\zz {{\mathbb{Z}}}
\def\rr {{\mathbb{R}}}
\def\cc {{\mathbb{C}}}

\def\C{\cc}
\def\R{\rr}
\def\Z {\zz}
\def\I {\mathcal{I}}
\def\sslash {/ \! /}
\def\rp {\mathbb{RP}}
\def\cp{\mathbb{CP}}

\def\D{\mathcal{D}}
\def\del{\partial}
\def\O{\mathcal{O}}
\def\E{\mathcal{E}}
\def\Q{\mathcal{Q}}

\def\g{\mathfrak{g}}
\def\Gad{G^{\operatorname{ad}}}
\def\HF {\mathit{HF}}

\def\HP{\mathit{HP}}

\def\HH{\mathbb{H}}
\def\Cat{\mathcal{C}}
\def\cs{\operatorname{c}}
\def\sl {{\operatorname{SL}(2, \cc)}}
\def\psl {{\operatorname{PSL}(2, \cc)}}
\def\sln {{\operatorname{SL}(n, \cc)}}
\def\gl{\operatorname{GL}(2, \cc)}
\def\su {{\operatorname{SU}(2)}}
\def\pu {{\operatorname{PU}(2)}}

\def\Sym{{\operatorname{Sym}}}

\def\st{\operatorname{st}}
\def\iso{\operatorname{iso}}
\def\Rep {R}
\def\k{\mathbf{k}}
\def\RepIrr{R_{\operatorname{irr}}}
\def\Rs {\mathscr{R}}
\def\RsIrr{\Rs_{\operatorname{irr}}}
\def\Ad {\operatorname{Ad}}
\def\im {\operatorname{im}}
\def\Char {X}
\def\Chars {\mathscr{X} \!}
\def\CharIrr {\Char_{\operatorname{irr}}}
\def\CharsIrr {\mathscr{X}_{\operatorname{irr}}}
\def\O {\mathcal{O}}

\def\P {\mathcal{P}}
\def\Reptw{\Rep_{\operatorname{tw}}}
\def\Xtw{\Char_{\operatorname{tw}}}

\def\Hom {\operatorname{Hom}}
\def\tr {\operatorname{Tr }}
\def\rk {\operatorname{rk }}

\def\ind{\operatorname{ind}}
\def\Stab{\mathscr{S}}
\def\stab{\operatorname{Stab}}
\def\can{\operatorname{can}}

\def\Perv {\operatorname{Perv}}

\def\Cb {\mathcal{C}^{\bullet}}

\def\irr{\operatorname{irr}}

\DeclareMathOperator{\sRHom}{\mathcal{RH}\hspace{-1pt}\mathit{om}}

\DeclareMathOperator{\Ext}{\operatorname{Ext}}

\def\HPcs{\HP_{\! \cs}}
\def\HPf{\HP_{\! \#}}
\def\HPfcs{\HP_{\! \#, \cs}}
\def\tw{\operatorname{tw}}

\def\DD{\mathbb{D}}
\def\L{\mathcal{L}}

\def\cH{\mathcal{H}}
\def\M{\mathcal{M}}
\def\tL{\tilde{L}}
\def\omegac{\omega_{\C}}

\def\Pb {\mathcal{P}^\bullet}
\def\Diff{\operatorname{Diff}}
\def\Qb{\mathcal{Q}^\bullet}
\def\id{\operatorname{id}}
\def\PVb {\mathcal{PV}^\bullet}
\def\red {\operatorname{red}}

\def\Crit{\operatorname{Crit}}
\def\Span{\operatorname{Span}}

\begin{document}

\title{A sheaf-theoretic model for $\sl$ Floer homology}

\author[Mohammed Abouzaid]{Mohammed Abouzaid}
\thanks {MA was supported by NSF grants DMS-1308179, DMS-1609148, and DMS-1564172, and by the Simons Foundation through its ``Homological Mirror Symmetry'' Collaboration grant.}
\address {Department of Mathematics, Columbia University, 2990 Broadway, MC 4406 \\
New York, NY 10027}
\email {abouzaid@math.columbia.edu}

\author[Ciprian Manolescu]{Ciprian Manolescu}
\thanks {CM was supported by NSF grants DMS-1402914 and DMS-1708320.}
\address {Department of Mathematics, UCLA, 520 Portola Plaza, Box 951555\\ Los Angeles, CA 90095}
\email {cm@math.ucla.edu}

\begin{abstract}
Given a Heegaard splitting of a three-manifold $Y$, we consider the $\sl$ character variety of the Heegaard surface, and two complex Lagrangians associated to the handlebodies. We focus on the smooth open subset corresponding to irreducible representations. On that subset, the intersection of the Lagrangians is an oriented d-critical locus in the sense of Joyce. Bussi associates to such an intersection a perverse sheaf of vanishing cycles. We prove that in our setting, the perverse sheaf is an invariant of $Y$, i.e., it is independent of the Heegaard splitting. The hypercohomology of the perverse sheaf can be viewed as a model for (the dual of) $\sl$ instanton Floer homology. We also present a framed version of this construction, which takes into account reducible representations. We give explicit computations for lens spaces and Brieskorn spheres, and discuss the connection to the Kapustin-Witten equations and Khovanov homology. 
\end {abstract}

\maketitle

\section{Introduction}

In \cite{Floer}, Floer associated to each homology three-sphere $Y$ an invariant $I_*(Y)$, called instanton homology. This is the homology of a complex generated by (perturbations of) irreducible flat $\su$ connections on $Y$, with the differential counting solutions to the $\su$ anti-self-dual (ASD) Yang-Mills equations on the cylinder $\R \times Y$. As shown by Taubes \cite{TaubesCasson}, the Euler characteristic of $I_*(Y)$ equals twice the Casson invariant from \cite{AMCasson}. The main motivation for instanton homology was to allow a definition of relative Donaldson invariants for four-manifolds with boundary; see \cite{DonaldsonBook} for results in this direction. Apart from this, instanton homology has had applications to three-manifold topology---most notably the proof of property P for knots by Kronheimer and Mrowka \cite{KMPropP}. 

Recently, there has been a surge of interest in studying the ASD equations with noncompact groups $\sl$ or $\psl$ instead of $\su$, as well as their topological twist, the Kapustin-Witten equations; cf. \cite{KapustinWitten, Taubes3, Taubes4}. In particular, Witten has a proposal for intepreting the Khovanov homology of knots or links in $\R^3$ in terms of solutions to a set of partial differential equations in five dimensions (usually called the Haydys-Witten equations); cf. \cite{FiveBranes, Haydys}. In this proposal, the Jones polynomial is recovered by counting solutions to the Kapustin-Witten equations on $\R^3 \times \R_+$, with certain boundary conditions; see \cite{FiveBranes, GaiottoWitten, WittenLectures1, WittenLectures2}.

In view of these developments, one would like to construct a variant of instanton Floer homology using the group $\sl$ instead of $\su$. In a sense, the $\sl$ case should be simpler than $\su$. For the unperturbed equations with complex gauge groups, physicists expect ``no instanton corrections'', i.e., no contributions to the Floer differential. Indeed, if there are only finitely many $\sl$ irreducible flat connections, and all are transversely cut out, then they must be in the same relative grading. In that case, the $\sl$ Floer homology could just be defined as the free Abelian group generated by those connections, in a single grading. However, for arbitrary $3$-manifolds, the moduli space (character variety) of $\sl$ flat connections can be higher dimensional, singular, and even non-reduced as a scheme. Furthermore, instanton corrections appear when we perturb the equations, and we run into difficult non-compactness issues. Thus, defining $\sl$ Floer homology directly using gauge theory seems challenging. 

Nevertheless, the lack of instanton corrections for the unperturbed equations indicates that $\sl$ Floer homology could be defined algebraically, without counting solutions to PDEs. The purpose of this paper is to use sheaf theory to give such a definition. 

Our construction draws inspiration from the Atiyah-Floer conjecture \cite{AtiyahFloer}. (See \cite{DaemiFukaya} for recent progress in the direction of this conjecture.) The Atiyah-Floer conjecture states that the  $\su$ instanton homology $I_*(Y)$ can be recovered as the Lagrangian Floer homology of two Lagrangians associated to a Heegaard decomposition for $Y$, with the ambient symplectic manifold being the moduli space of flat $\su$ connections on the Heegaard surface $\Sigma$. In a similar fashion, we consider the moduli space $\Char(\Sigma)$ of flat $\sl$ connections on $\Sigma$ (or, equivalently, representations of $\pi_1(\Sigma)$ into $\sl$). The space $\Char(\Sigma)$ is called the {\em character variety} of $\Sigma$. It contains an open set $\CharIrr(\Sigma) \subset \Char(\Sigma)$ corresponding to irreducible flat connections. For the three-manifold $Y$, we can define $\Char(Y)$ and $\CharIrr(Y)$ in an analogous way.  The space $\CharIrr(\Sigma)$ is a smooth, complex symplectic manifold. Inside $\CharIrr(\Sigma)$ we have two complex Lagrangians $L_0$ and $L_1$, associated to the two handlebodies.  The intersection $L_0 \cap L_1$ is isomorphic to $\CharIrr(Y)$; cf. Lemma~\ref{lem:idab} (a) below.

We could try to take the Lagrangian Floer homology of $L_0$ and $L_1$ inside $\CharIrr(\Sigma)$, but non-compactness issues appear here just as in the gauge-theoretic context. Instead, we make use of the structure of $\CharIrr(Y)$ as a derived scheme. Joyce \cite{Joyce} introduced the theory of d-critical loci, which is a way of encoding some information from derived algebraic geometry in terms of classical  data. The intersection of two algebraic Lagrangians in an algebraic symplectic manifold is a d-critical locus; see \cite[Corollary 2.10]{PTVV} and \cite[Corollary 6.8]{BBJ}. If the Lagrangians come equipped with spin structures, the d-critical locus gets an orientation in the sense of \cite[Section 2.5]{Joyce}. Furthermore, to any oriented d-critical locus one can associate a perverse sheaf of vanishing cycles, cf. \cite{BBDJS}; in the case of an algebraic Lagrangian intersection, the hypercohomology of this sheaf is conjectured to be the same as the Lagrangian Floer cohomology; see \cite[Remark 6.15]{BBDJS}. Furthermore, in the complex analytic context, Bussi \cite{Bussi} gave a simpler way of constructing the perverse sheaf for complex Lagrangian intersections. 

In our setting, we apply Bussi's construction to the Lagrangians $L_0, L_1 \subset \CharIrr(\Sigma)$. The resulting perverse sheaf of vanishing cycles is denoted $P^{\bullet}_{L_0, L_1}$. Our main result is:

\begin{theorem}
\label{thm:main1}
Let $Y$ be a closed, connected, oriented three-manifold. Then, the object $P^{\bullet}(Y):=P^{\bullet}_{L_0, L_1}$ (constructed from a Heegaard decomposition, as above) is an invariant of the three-manifold $Y$, up to canonical isomorphism in a category of perverse sheaves $\Perv'(\CharIrr(Y))$.

As a consequence, its hypercohomology
$$  \HP^*(Y) := \HH^*( P^{\bullet}(Y))$$
is also an invariant of $Y$, well-defined up to canonical isomorphism in the category of $\Z$-graded Abelian groups. 
\end{theorem}

The content of Theorem~\ref{thm:main1} is that $P^{\bullet}_{L_0, L_1}$ is independent of the Heegaard decomposition used to construct it. The proof requires checking invariance under a stabilization move, as well as a naturality result similar to that proved by Juh\'asz, Thurston and Zemke in \cite{JuhaszThurston}, for Heegaard Floer homology. Naturality means that as we relate a Heegaard diagram to another by a sequence of moves, the induced isomorphism is independent of the sequence we choose. Moreover, we want the diffeomorphism group of $Y$ to act on our invariant $P^{\bullet}_{L_0, L_1}$. Since the diffeomorphism group can act non-trivially on $\CharIrr(Y)$ itself, we cannot simply view $P^{\bullet}_{L_0, L_1}$ as an object in the usual category of perverse sheaves $\Perv(\CharIrr(Y))$, where the morphisms cover the identity on $\CharIrr(Y)$. Rather, we use a slightly different category $\Perv'(\CharIrr(Y))$, which will be introduced in Definition~\ref{def:pervp} below.

\begin{remark}
\label{rem:PTVV}
One can construct the perverse sheaves $P^{\bullet}(Y)$ more directly, without Heegaard decompositions, by resorting to the theory of shifted symplectic structures in derived algebraic geometry developed by Pantev-To\"en-Vaqui\'e-Vezzosi \cite{PTVV}. In this paper, we preferred to use the methods from \cite{Bussi} since they are more concrete, and make computations easier. In particular, they do not require any knowledge of derived algebraic geometry.
\end{remark}

We call $\HP^*(Y)$ the {\em sheaf-theoretic $\sl$ Floer cohomology} of $Y$. If an $\sl$ Floer cohomology for $Y$ can be defined (using either gauge theory or symplectic geometry), we conjecture that $\HP^*(Y)$ would be isomorphic to it.

Note that, whereas $\su$ instanton homology is only defined for integer homology spheres, the invariant $\HP^*(Y)$ is defined for all closed, connected, oriented three-manifolds.

We call the Euler characteristic
$$ \lambda^{P}(Y) := \chi(\HP^*(Y))$$
 the {\em full (sheaf-theoretic) $\sl$ Casson invariant} of $Y$. We use the name {\em full} to distinguish it from the $\sl$ Casson invariant defined by Curtis in \cite{Curtis}, which counts only the isolated irreducible flat connections.
 
Our construction of $\HP^*(Y)$ has some limitations too, because it only involves irreducible flat connections. In the $\su$ context, one theory that takes into account the reducibles is the {\em framed instanton homology} $\mathit{FI}_*(Y)$ considered by Kronheimer and Mrowka in \cite{KMknots}. This is defined for any three-manifold $Y$, and its construction uses connections in an admissible $\pu$ bundle over $Y \# T^3$. Framed instanton homology was further studied by Scaduto in \cite{Scaduto}, where it is denoted $I^\#(Y)$. Moreover, symplectic counterparts to framed instanton homology were defined in \cite{WWFloerField} and \cite{MWextended}.

Consider a Heegaard decomposition of a three-manifold $Y$, as before. Following Wehrheim and Woodward \cite[Section 4.4]{WWFloerField}, we take the connected sum of the Heegaard surface $\Sigma$ (near a basepoint $z$) with a torus $T^2$, and obtain a higher genus surface $\Sigma^\#$. On $\Sigma^\#$ we consider the moduli space of {\em twisted} flat $\sl$ connections, $\Char_{\tw}(\Sigma^{\#})$, which is a smooth complex symplectic manifold. There are smooth Lagrangians $L_0^{\#} , L_1^{\#} \subset  \Char_{\tw}(\Sigma^{\#})$ coming from the two handlebodies. Their intersection is the representation variety $\Rep(Y) := \Hom(\pi_1(Y), \sl)$. Bussi's construction yields a perverse sheaf of vanishing cycles $P^{\bullet}_{L^{\#}_0, L^{\#}_1}$ over $\Rep(Y)$.

\begin{theorem}
\label{thm:framed}
Let $Y$ be a closed, connected, oriented three-manifold, and $z \in Y$ a basepoint. Then, the object $P^{\bullet}_\#(Y,z):=P^{\bullet}_{L_0^\#, L_1^\#}$ is an invariant of the three-manifold $Y$ and the basepoint $z$, up to canonical isomorphism in a category of perverse sheaves $\Perv'(\Rep(Y))$.

As a consequence, its hypercohomology
$$  \HPf^*(Y,z) := \HH^*( P^{\bullet}_{\#}(Y, z))$$
is also an invariant of $(Y, z)$, well-defined up to canonical isomorphism in the category of $\Z$-graded Abelian groups. 
\end{theorem}

We call  $\HPf^*(Y, z)$ the {\em framed sheaf-theoretic $\sl$ Floer cohomology} of $Y$. When we are only interested in its isomorphism class, we will drop $z$ from the notation and write $\HPf^*(Y)$ for $\HPf^*(Y, z)$.

To compute the invariants defined in this paper, the main tool we use is the following.

\begin{theorem}
\label{thm:main2}
Let $Y$ be a closed, connected, oriented three-manifold, $\Rep(Y)$ its $\sl$ representation variety, and $\Char(Y) = \Rep(Y)\sslash \psl$ its character variety, with the open subset $\CharIrr(Y) \subset \Char(Y)$ consisting of irreducibles. We also let $\Rs(Y)$ be the corresponding representation scheme, and $\CharsIrr(Y) \subset \Chars(Y)$ the character scheme. Let $z \in Y$ be a basepoint.

(a) If $\CharsIrr(Y)$ is regular, then $P^{\bullet}(Y)$ is a (degree shifted) local system on $\CharIrr(Y)$, with stalks isomorphic to $\Z$.

(b) If $\Rs(Y)$ is regular, then $P^{\bullet}_{\#}(Y, z)$ is a (degree shifted) local system on $\Rep(Y)$, with stalks isomorphic to $\Z$. 
\end{theorem}

In some situations, we can show that the local systems appearing in Theorem~\ref{thm:main2} are trivial. This allows us to do concrete calculations for various classes of three-manifolds. We give a few examples below, with $\Z_{(0)}$ denoting the group $\Z$ in degree $0$.

\begin{theorem}
\label{thm:lens}
Consider the lens space $L(p, q)$ with $p$ and $q$ relatively prime. Then $\HP^*(L(p,q)) =0$ and
$$ \HPf^*(L(p,q)) \cong
\begin{cases}
\Z_{(0)} \oplus H^{*+2}(S^2; \Z)^{\oplus (p-1)/2} & \text{ if $p$ is odd,}\\ 
 \Z_{(0)}^{\oplus 2} \oplus H^{*+2}(S^2; \Z)^{\oplus (p-2)/2} & \text{ if $p$ is even.}
 \end{cases}$$
\end{theorem}

\begin{theorem}
\label{thm:Brieskorn}
For the Brieskorn spheres $\Sigma(p,q,r)$ with $p,q,r$ pairwise relatively prime, we have
$$ \HP^*(\Sigma(p,q,r)) \cong \Z^{\oplus (p-1)(q-1)(r-1)/4}_{(0)}$$
and
$$ \HPf^*(\Sigma(p,q,r)) \cong \Z_{(0)} \oplus H^{*+3}(\rp^3; \Z)^{\oplus (p-1)(q-1)(r-1)/4}.$$
\end{theorem}

Let $Y$ be a closed oriented $3$-manifold. Recall that a smoothly embedded surface $S \subset Y$ is called {\em incompressible} if there is no disk $D$ embedded in $M$ such that $D \cap S = \del D$ and $\del D$ does not bound a disk in $S$. The manifold $Y$ is called {\em sufficiently large} if it contains a properly embedded, two-sided, incompressible surface. (Haken manifolds are those that are sufficiently large and irreducible.) By the work of Culler and Shalen \cite{CullerShalen}, when $Y$ is not sufficiently large, the character variety $\CharIrr(Y)$ has only zero-dimensional components; compare \cite[Proposition 3.1]{Curtis}. From here we easily obtain the following result.

\begin{theorem}
\label{thm:NSL}
For three-manifolds $Y$ that are not sufficiently large, the invariant $\HP^*(Y)$ is supported in degree zero.
\end{theorem}

We also have the following relation between our invariants and the Heegaard genus, which was pointed out to us by  Ikshu Neithalath.

\begin{theorem}
\label{thm:ikshu}
If $Y$ admits a Heegaard splitting of genus $g$, then
$$ \HP^k(Y) \neq 0 \ \Rightarrow -3g+3 \leq k  \leq 3g-3$$
and
$$ \HP^k_{\#}(Y) \neq 0 \ \Rightarrow -3g \leq  k \leq 0.$$
\end{theorem}

Character varieties of $\sl$ representations play an important role in three-dimensional topology, for example in the paper \cite{CullerShalen} mentioned above, in the work of Morgan and Shalen \cite{MorganShalen1, MorganShalen2, MorganShalen3}, and in the proof of the cyclic surgery theorem by Culler, Gordon, Luecke and Shalen \cite{CGLS}. It would be interesting to explore if there are more connections between $\HP^*$ and classical three-manifold topology, beyond Theorems~\ref{thm:NSL} and \ref{thm:ikshu}.

The organization of the paper is as follows. In Section~\ref{sec:Varieties} we gather a few facts about representation and character varieties. In Section~\ref{sec:Motivation} we introduce the complex  Lagrangians $L_0, L_1, L_0^{\#}, L_1^{\#}$, and present in more detail the motivation coming from the Atiyah-Floer conjecture.  Section~\ref{sec:sheaves} contains a review of Bussi's construction of perverse sheaves associated to complex Lagrangian intersections. In Section~\ref{sec:prop} we discuss the behavior of Bussi's perverse sheaf under stabilization, and in Section~\ref{sec:Clean} we study the perverse sheaf in the case where the Lagrangians intersect cleanly. In Section~\ref{sec:Invariant} we define our invariants and prove Theorems~\ref{thm:main1} and \ref{thm:framed}.  Section~\ref{sec:Examples} contains the proofs of Theorems~\ref{thm:main2}, ~\ref{thm:lens} and \ref{thm:Brieskorn}, together with a few other calculations. In Section~\ref{sec:Further} we describe further directions for research, and connections to other fields.

\medskip
\textbf{Acknowledgements.} We have benefited from discussions with Ian Agol, Hans Boden, Francis Bonahon, Marco Castronovo, Ben Davison, Michael Kapovich, Yank{\i} Lekili, Sam Lewallen, Jake Rasmussen, Rapha\"el Rouquier, Nick Rozenblyum, Pierre Schapira, Paul Seidel, Adam Sikora, Ivan Smith, Edward Witten, and Chris Woodward. We are particularly indebted to Dominic Joyce for explaining to us his work. We would also like to thank Laurent C\^ot\'e, Ikshu Neithalath, and the referees, for their comments on a previous version of this paper.

\section{Representation varieties and character varieties}
\label{sec:Varieties}

In this section we gather some facts about representations of finitely generated groups into $\sl$, as well as examples. We recommend the books  \cite{LubotzkyMagid}, \cite{KapovichBook} and the articles \cite{Goldman}, \cite{CullerShalen}, \cite{Heusener}, \cite{Sikora}
for more details about this topic.

Throughout the paper (except where otherwise noted, in Section~\ref{sec:otherG}), we let $G$ denote the group $\sl$, with Lie algebra $\g= \mathfrak{sl}(2, \C)$ and center $Z(G) = \{\pm I\}$. We let $\Gad = G/Z(G)=\psl$ be the adjoint group of $G$.

We denote by $B \subset G$ the Borel subgroup of $G$ consisting of upper-triangular matrices, and by $D$ the subgroup consisting of diagonal matrices. We also let $B_{P} \subset B$ be the subgroup of $B$ consisting of parabolic elements, i.e. those of the form $\pm \begin{pmatrix} 1 & a \\ 0 & 1 \end{pmatrix}$, with $a \in \C$. Note that $D$ and $B_P$ are both Abelian, with intersection $D \cap B_P = Z(G)$.

\subsection{Representation varieties}
\label{sec:varieties}
Let $\Gamma$ be a finitely generated group. Its {\em representation variety} is defined as
$$  \Rep(\Gamma) = \Hom(\Gamma, G).$$
If $\Gamma$ has $k$ generators, by viewing $G$ as a subset of $\text{GL}(2,\C)\cong \C^4$ we find that $\Rep(\Gamma)$ is an affine algebraic subvariety of $\C^{4k}$. Indeed, the relations in $\Gamma$, together with the determinant one conditions, produce a set of polynomial equations in $4k$ variables,
$$ f_i(x_1, \dots, x_{4k})=0,$$
so that their common zero set is $\Rep(\Gamma)$. Here, the subscripts $i$ take values in some index set $\I$.

We can also consider the {\em representation scheme}
\begin{equation}
\label{eq:Spec}
 \Rs(\Gamma) = \operatorname{Spec} \bigl( \C[x_1, \dots, x_{4k}] / (f_i)_{i \in \I} \bigr).
 \end{equation}
The affine scheme $\Rs(\Gamma)$ is independent of the presentation of $\Gamma$, up to canonical isomorphism. The scheme $\Rs(\Gamma)$ may be non-reduced; the corresponding reduced scheme gives the variety $\Rep(\Gamma)$.

The group $\Gad$ acts on $\Rep(\Gamma)$ by conjugation. Given a representation $\rho: \Gamma \to G$, we denote by $\stab(\rho) \subseteq \Gad$ its stabilizer, and by $\O_{\rho} \cong \Gad/\stab(\rho)$ its orbit.

We distinguish five kinds of representations $\rho: \Gamma \to G$:
\begin{enumerate}[(a)]
\item {\em irreducible}, those such that the corresponding representation on $\C^2$ does not preserve any line; in other words, those that are not conjugate to a representation into the Borel subgroup $B$. An irreducible representation has trivial stabilizer. Its orbit is a copy of $\Gad=\psl$, which is topologically $\rp^3 \times \R^3$;\smallskip

\item {\em non-Abelian reducible}, those that are conjugate to a representation with image in $B$, but not into one with image in $B_{P}$ or $D$. Such representations have trivial stabilizer also;\smallskip

\item {\em parabolic non-central}, those that are conjugate to a representation with image in $B_{P}$, but not in $\{\pm I\}$. Such representations have stabilizer $B_{P}/\{\pm I\} \cong \C$. Their orbit $\O_{\rho} \cong G/B_P$ is a bundle over $G/B \cong \mathbb{CP}^1$ with fiber $B/B_P \cong \C^*$. In fact, $\O_{\rho}$ is diffeomorphic to $\rp^3 \times \R$;\smallskip

\item {\em diagonal non-central}, those that are conjugate to a representation with image in $D$, but not in $\{\pm I\}$. Such representations have stabilizer $D/\{\pm I\} \cong \C^*$. Their orbit is a copy of $\Gad/\C^*$, topologically $TS^2$;\smallskip

\item {\em central}, those with image in $Z(G)=\{\pm I\}$. Their stabilizer is the whole group $\Gad$, and their orbit is a single point. \smallskip
\end{enumerate}

Representations of types (b)-(e) are called {\em reducible}. Those of types (a), (d) and (e) are {\em completely reducible}, or {\em semi-simple}. Those of types (c), (d) and (e) have Abelian image, and we call them {\em Abelian}.

We will denote by $\RepIrr(\Gamma) \subset \Rep(\Gamma)$ the (Zariski open) subset consisting of irreducible representations, and similarly by $\RsIrr(\Gamma) \subset \Rs(\Gamma)$ the open subscheme associated to irreducibles. (For a proof of openness, see for example \cite[Proposition 27]{Sikora}.)

Given a representation $\rho: \Gamma \to G$, we denote by $\Ad \rho := \Ad \circ \rho$ the associated adjoint representation on $\g$.  A map $\xi: \Gamma \to\g$ is called a {\em 1-cocycle} if
\begin{equation}
\label{eq:1cocycle}
 \xi(xy) = \xi(x)+\Ad_{\rho(x)} \xi(y), \ \ \text{for all } x, y \in \Gamma.
 \end{equation}
Further, $\xi$ is a {\em 1-coboundary} if it is of the form
$$ \xi(x) = u - \Ad_{\rho(x)} u$$
for some $u \in \g$. The space of 1-cocycles is denoted $Z^1(\Gamma; \Ad \rho)$ and the space of 1-coboundaries is denoted $B^1(\Gamma; \Ad \rho)$. Their quotient is the group cohomology 
$$H^1(\Gamma; \Ad \rho) = Z^1(\Gamma; \Ad \rho)/B^1(\Gamma; \Ad \rho)$$
When $\Gamma =\pi_1(M)$ for a topological space $M$, we can identify $H^1(\Gamma; \Ad \rho)$ with $H^1(M; \Ad \rho)$, the first cohomology of $M$ with coefficients in the local system given by $\Ad \rho$.

By a result of Weil \cite{Weil}, the Zariski tangent space to the scheme $\Rs(\Gamma)$ at a closed point $\rho$ is identified with $Z^1(\Gamma; \Ad \rho)$. We can also consider the (possibly smaller) Zariski tangent space to the variety $\Rep(\Gamma)$. In general, we have a chain of inequalities
\begin{equation}
\label{eq:chain}
\dim \O_{\rho} = \dim B^1(\Gamma; \Ad \rho) \leq \dim_{\rho} \Rep(\Gamma) \leq \dim T_{\rho}\Rep(\Gamma) \leq \dim T_{\rho}\Rs(\Gamma) = \dim Z^1(\Gamma; \Ad \rho),
\end{equation}
where $\dim_{\rho}$ denotes the local dimension at $\rho$. Compare \cite[Chapter 2]{LubotzkyMagid} and \cite[Lemma 2.6]{Heusener}. 

Following \cite{Heusener} and \cite{Sikora}, we say:
\begin{definition}
\label{def:redreg}
$(a)$  The representation $\rho$ is called {\em reduced} if $\dim T_{\rho}\Rep(\Gamma) = \dim Z^1(\Gamma; \Ad \rho)$ i.e., the last inequality in \eqref{eq:chain} is an equality.
This is the same as asking for the scheme $\Rs(\Gamma)$ to be reduced at $\rho$.

$(b)$ The representation $\rho$ is called {\em regular} (or {\em scheme smooth}) if  $\dim_{\rho} \Rep(\Gamma) =\dim Z^1(\Gamma; \Ad \rho)$, i.e., the last two inequalities in \eqref{eq:chain} are equalities. This is the same as asking for the scheme $\Rs(\Gamma)$ to be regular (i.e., smooth) at $\rho$.
\end{definition}

Note that if $H^1(\Gamma; \Ad \rho)=0$, then from \eqref{eq:chain} we see that $\rho$ is regular. In fact, in that case, any representation sufficiently close to $\rho$ is actually conjugate to $\rho$; see \cite{Weil}.

\subsection{Character varieties}
Let us consider again the action of $\Gad$ on the representation variety $\Rep(\Gamma)$. The {\em character variety} of $\Gamma$ is defined to be the categorical quotient
$$ \Char(\Gamma) = \Rep(\Gamma) \sslash \Gad.$$
If we let $\Rep'(\Gamma) \subset \Rep(\Gamma)$ denote the subset consisting of completely reducible representations, the categorical quotient can be constructed explicitly as 
$$ \Char(\Gamma) =\Rep'(\Gamma)/\Gad.$$
See \cite[Theorem 1.27]{LubotzkyMagid} or \cite[Section 7]{Sikora}. 

There is also a {\em representation scheme} 
$$\Chars(\Gamma) = \Rs(\Gamma) \sslash \Gad.$$
In terms of the notation in \eqref{eq:Spec}, we have
$$\Chars(\Gamma) =   \operatorname{Spec} \bigl( \C[x_1, \dots, x_{4k}] / (f_i)_{i \in \I}  \bigr)^{\Gad}.$$
The reduced scheme associated to $\Chars(\Gamma)$ is the character variety $\Char(\Gamma)$. See \cite[Section 12]{Sikora} for more details.

We denote by $\CharIrr(\Gamma) = \RepIrr(\Gamma)/\Gad  \subset \Char(\Gamma)$ the open subvariety made of classes of irreducible representations. Similarly, there is an open subscheme $\CharsIrr(\Gamma)$ of $\Chars(\Gamma)$, corresponding to irreducible representations.

By \cite[Corollary 1.33]{LubotzkyMagid}, the conjugacy class of a completely reducible representation $\rho \in \Rep'(\Gamma)$ is determined by its character, 
$$\chi_{\rho}: \Gamma \to \C, \ \ \chi_{\rho}(g) = \tr(\rho(g)).$$
For each $g \in G$, we can define a regular function 
\begin{equation}
\label{eq:taug}
\tau_g: R(\Gamma) \to \C, \ \tau_g( \rho) = \chi_{\rho}(g).
\end{equation}
Let $T$ be the ring generated by the functions $\tau_g$; this is the coordinate ring of $ \Char(\Gamma)$; cf. \cite[1.31]{LubotzkyMagid}. Using the identities
$$ \tau_{g} \tau_h = \tau_{gh} + \tau_{gh^{-1}},$$ 
one can prove that, if $g_1, g_2, \dots, g_n$ are generators of $\Gamma$, then the $2^n-1$ functions 
$$ \tau_{g_{i_1} \dots g_{i_k}}, \ 1 \leq k \leq n, \ 1 \leq i_1 < \dots < i_k \leq n,$$  generate $T$. This gives a closed embedding of $\Char(\Gamma)$ into an affine space $\C^N$, where $N = 2^n-1$. See \cite[Proposition 1.4.1 and Corollary 1.4.5]{CullerShalen} and \cite[Proposition 4.4.2]{ShalenSurvey}.

With regard to tangent spaces, we have the following:
\begin{proposition}[cf. Theorems 53 and 54 in \cite{Sikora}] 
\label{prop:TC}
$(a)$ Let $\rho \in \Rep(\Gamma)$ be a completely reducible representation. Then, the projections 
$\Rep(\Gamma) \to \Char(\Gamma)$ and  $\Rs(\Gamma) \to \Chars(\Gamma)$ induce natural linear maps
$$ \phi: T_{\rho} \Rep(\Gamma)/B^1(\Gamma; \Ad \rho) \to T_{[\rho]}\Char(\Gamma)$$
and
$$ \Phi: H^1(\Gamma; \Ad \rho) \to T_{[\rho]}\Chars(\Gamma).$$

$(b)$ If $\rho$ is irreducible, then $\phi$ and $\Phi$ are isomorphisms.

$(c)$ If $\rho$ is completely reducible and regular, then
$$ \dim T_0(H^1(\Gamma; \Ad \rho) \sslash \stab(\rho)) \cong T_{[\rho]} \Char(\Gamma) = T_{[\rho]} \Chars(\Gamma),$$
where we considered the natural action of $\stab(\rho)$ on group cohomology.
\end{proposition}

\begin{proposition}[cf. Corollary 55  in \cite{Sikora}]
\label{prop:red}
 An irreducible representation $\rho \in \Rep(\Gamma)$ is reduced if and only if the scheme $\Chars(\Gamma)$ is reduced at $[\rho]$.
\end{proposition}

We refer to Sikora's paper \cite{Sikora} for more details. The results are stated there for good representations into a reductive algebraic group $G$.  (See Section~\ref{sec:otherG} for the definition of good.) In the case $G=\sl$, all irreducible representations are good.

We also have the following fact:
\begin{lemma}
\label{lem:reg}
 An irreducible representation $\rho \in \Rep(\Gamma)$ is regular if and only if the scheme $\Chars(\Gamma)$ is regular at $[\rho]$. 
\end{lemma}

\begin{proof}
The ``only if'' part is Lemma 2.18 in \cite{LubotzkyMagid}. For the ``if'' part, note that if $\Chars(\Gamma)$ is regular at $[\rho]$, it is regular in a neighborhood $U$ of $[\rho]$. The  neighborhood $U$ may be chosen to consist of irreducibles. We conclude that a neighborhood of $\rho$ in the representation scheme $\Rs(\Gamma)$ is a $\Gad$-bundle over $U$, which is smooth. Hence, $\rho$ is regular; cf. Definition~\ref{def:redreg} (b).
\end{proof}

\begin{remark}
If $\rho$ is regular but reducible, then $\Chars(\Gamma)$ may not be regular at $[\rho]$. See Section~\ref{sec:free} below, namely the case where $\Gamma$ is a free group with at least three generators.
\end{remark}

\subsection{The case of free groups}
\label{sec:free}
We now specialize to the case where $\Gamma =F_k$, the free group on $k$ variables. The representations of free groups into $\sl$ have been studied extensively in the literature; see for example \cite{Horowitz}, \cite{HeusenerPorti}.

We have $\Rep(F_k) \cong G^{k}$, and all representations are regular. When $k=1$, the representations are Abelian, and they can be central, diagonal non-central, or parabolic non-central. For $k \geq 2$, we find representations of all possible types. For example, one obtains a non-Abelian reducible representation of $F_2$ by sending one generator to a non-central diagonal matrix, and the other to a non-central parabolic matrix. This works as well for $F_k$ for $k > 2$ by simply sending all the other generators to $I$.

With regard to the character variety $X(F_k)$:
\begin{itemize}
\item When $k=1$, let $g$ be the generator of $F_1$. We then have $X(F_1) \cong \C$, with the coordinate being the trace $\tau_g$, in the notation \eqref{eq:taug};
\item When $k=2$, let $g$ and $h$ be the generators of $F_k$. We have $X(F_2) \cong \C^3$, with the three coordinates being $x=\tau_g, y=\tau_h$ and $z=\tau_{gh}$. By a result of Fricke \cite{Fricke} and Vogt \cite{Vogt},  the reducible locus $\Char(F_2) \setminus \CharIrr(F_2)$ is the hypersurface given by the equation $x^2 + y^2 + z^2 - xyz=4$. See \cite{GoldmanFV} for an exposition of this; 
\item For $k \geq 3$, the character variety is singular, and its singular locus is exactly the reducible locus, $\Char(F_k) \setminus \CharIrr(F_k)$; see \cite[Section 5.3]{HeusenerPorti}. The fact that all irreducible representations are regular can be seen from Lemma~\ref{lem:reg}. The variety $X(F_k)$ has complex dimension $3k-3$, and its reducible locus has dimension $k$. 
\end{itemize}

For future reference, we note the following facts about the topology of $\CharIrr(F_k)$.
\begin{lemma}
\label{lem:pifree}
For $k \geq 3$, we have $\pi_1(\CharIrr(F_k)) = 1$ and $\pi_2(\CharIrr(F_k))=\Z/2$. Hence, we have $H^1(\CharIrr(F_k); \Z/2) =0$ and  $H^2(\CharIrr(F_k); \Z) =0.$
\end{lemma}

\begin{proof}
For $k\geq 3$, consider the reducible locus of the representation variety, $Z = \Rep(F_k) \setminus \RepIrr(F_k)$. Any reducible representation fixes a line in $\C^2$; once we choose the line, we can assume the representation is upper triangular, i.e. takes values in $B \subset G$. Since $B$ has complex dimension $2$, we find that $Z$ has dimension $2k+1$. (The extra degree of freedom comes from choosing the line.) Since $\Rep(F_k) \cong G^k$, we see that $Z$ is of codimension $k-1$, which means real codimension at least $4$. Hence, removing $Z$ from $\Rep(F_k)$ does not change $\pi_1$ and $\pi_2$. From the polar decomposition we see that $G$ is diffeomorphic to $T\su \cong TS^3 \cong S^3 \times \R^3$, which has $\pi_1 = \pi_2 = 1$. We deduce that $\pi_1(\RepIrr(F_k)) =  \pi_2(\RepIrr(F_k))=0$.

We now look at the long exact sequence for the homotopy groups of the fibration $$\Gad \hookrightarrow \RepIrr(F_k) \twoheadrightarrow \CharIrr(F_k).$$
Since $\Gad$ is diffeomorphic to $T\operatorname{SO}(3) \cong \rp^3 \times \R^3$, we obtain that $\pi_1(\CharIrr(F_k)) = 1$ and $\pi_2(\CharIrr(F_k))\cong \pi_1(\rp^3) = \Z/2$.

The computations for cohomology come from the Hurewicz theorem and the universal coefficients theorem.
\end{proof}

\begin{remark}
\label{rem:doi}
When $k = 2$, we can view $\CharIrr(F_2)$ as the complement of the hypersurface $w( x^2w + y^2w + z^2w - xyz-4w^3)=0$ in $\mathbb{CP}^3$. By \cite[Ch.4, Proposition 1.3]{DimcaSing}, we get $H_1(\CharIrr(F_2); \Z) \cong \Z$, so the fundamental group is nontrivial.
\end{remark}

\subsection{Examples for three-manifolds}
\label{sec:ex3}
In this section we will give a few examples of representation and character varieties coming from fundamental groups of three-manifolds $Y$. In general, when $Y$ is a manifold, we will write $\Rep(Y)$ for $\Rep(\pi_1(Y))$, and similarly with $\Rs(Y), \Char(Y)$, etc. 

\begin{remark}
Note that $\pi_1(Y)$ and hence $\Rep(Y), \Rs(Y)$ are defined after choosing a basepoint $z \in Y$. A different choice of basepoint induces a (non-canonical) isomorphism between the respective objects. However, we will drop $z$ from notation for convenience. In the case of character varieties and character schemes, since we divide out by conjugation, the isomorphisms induced by the change of basepoint are actually natural.
\end{remark}

In Examples~\ref{ex:S3}-\ref{ex:T3} below, both the character and representation schemes are reduced, as can be checked using Definition~\ref{def:redreg}(a) and Proposition~\ref{prop:red}. In view of this, we will focus on describing the varieties $R(Y)$ and $X(Y)$.

\begin{example}
\label{ex:S3}
When $Y=S^3$, we have that $\pi_1(Y)$ is trivial, so both $\Rep(Y)$ and $\Char(Y)$ consist of a single point.
\end{example}

\begin{example}
\label{ex:free}
Let $Y$ be the connected sum of $k$ copies of $S^1 \times S^2$. Then $\pi_1(Y)$ is the free group $F_k$ on $k$ generators. The varieties $\Rep(F_k)$ and $\Char(F_k)$ were discussed in Section~\ref{sec:free}.
\end{example}

\begin{example}
\label{ex:Lpq}
Let $Y$ be the lens space $L(p,q)$ with $p > 0$ and $\gcd(p,q) = 1.$ Then $\pi_1(L(p,q)) = \Z/p$. A representation $\rho: \Z/p \to \sl$ is determined by what it does on the generator $[1]\in \Z/p$; up to conjugacy, it must send it to a diagonal matrix of the form $\textit{diag}(u, u^{-1})$, where $u$ is a $p^{\text{th}}$ root of unity. Note that $\textit{diag}(u, u^{-1})$ is conjugate to $\textit{diag}(u^{-1}, u)$. Thus, in terms of the list of representation types in Section~\ref{sec:varieties}:
\begin{itemize}
\item If $p$ is odd, then $R(Y)$ consists of $(p+1)/2$ conjugacy classes of diagonal representations, one being the trivial representation and the rest all non-central.  Thus, $R(Y)$ is the disjoint union of a point and $(p-1)/2$ copies of $TS^2$, and $X(Y)$ consists of $(p+1)/2$ points.
\item If $p$ is even, then $R(Y)$ consists of $(p/2)+1$ conjugacy classes of diagonal representations, two being central and the rest non-central.  Thus, $R(Y)$ is the disjoint union of two points and $(p-2)/2$ copies of $TS^2$, and $X(Y)$ consists of $(p/2)+1$ points.
\end{itemize}
Furthermore, all representations are regular. Indeed, we claim that $H^1(\Z/p; \Ad \rho)=0$ for any such $\rho$. In general, the first cohomology of the cyclic group $\Z/p$ with values in a module $M$ is 
\begin{equation}
\label{eq:h1m}
H^1(\Z/p; M) = \{ m \in M \mid (1+\zeta + \zeta^2 + \dots + \zeta^{p-1})m=0\} / \{(1-\zeta)m \mid m \in M\},
\end{equation}
where $\zeta$ is the action of the generator. In our case, $\zeta$ is conjugation by the matrix $A=\rho([1])$, and $m$ is a traceless $2$-by-$2$ matrix. If $A=\pm I$ then clearly the right hand side of \eqref{eq:h1m} is zero. If $A \sim \textit{diag}(u, u^{-1}) \neq \pm I$, then any element of $\g$ can be written as the commutator $[m, A]$ for some $m \in \g$. Hence, $(1-\zeta)m = [m, A] A^{-1}$ can be any element of $\g$, and we again find that the right hand side of \eqref{eq:h1m} is zero.
\end{example}

\begin{example}
\label{ex:Br3}
Let $Y$ be the Brieskorn sphere
$$\Sigma(p,q,r) = \{(x, y, z) \in \C^3 \mid x^p + y^q + z^r = 0 \} \cap S^5,$$
where $p, q, r > 0$ are pairwise relatively prime integers. The representations of $\pi_1(\Sigma(p,q,r))$ into $\sl$ were studied by Boden and Curtis in \cite[Section 3]{BodenCurtis}. There is the trivial representation and 
$$N = \frac{(p-1)(q-1)(r-1)}{4}$$
irreducible ones. The first cohomology $H^1(\Sigma(p,q,r); \Ad \rho)$ vanishes for all these representations, by \cite[Lemma 2.4]{BodenCurtis}, so they are all regular. 

Therefore, $\Rep(Y)$ consists of one point and $N$ copies of $\psl \cong \rp^3 \times \R^3,$ and $\Char(Y)$ consists of $N+1$ points. 
\end{example}

\begin{example}
\label{ex:Br4}
More generally, let $Y=\Sigma(a_1, \dots, a_n)$ be a Seifert fibered homology sphere, where $a_1, \dots, a_n > 0$ are pairwise relatively prime. We can arrange so that $a_i$ is odd for $i \geq 2$. The representations of $\pi_1(Y)$ into $\sl$ were  studied in \cite[proof of Theorem 2.7]{BodenCurtis}. There is the trivial representation and some irreducibles, which come in families. Precisely, the character variety $\Char(Y) = \mathit{pt} \cup \CharIrr(Y)$ is regular, with $\CharIrr(Y)$ being the disjoint union of components $\M_{\alpha}$, one for each $\alpha=(\alpha_1,\dots, \alpha_n)$, with
$$ \alpha_1 = k_1/2a_1, \ k_1 \in \Z, \ 0 \leq k_1 \leq a_1,$$
$$ \alpha_i = k_i/a_i, \ k_i \in \Z, \ 0 \leq k_i < a_i /2 \ \text{for } i \geq 2.$$
Each $\M_{\alpha}$ can be identified with the moduli space of parabolic Higgs bundles of parabolic degree zero over $\mathbb{CP}^1$ with $n$ marked points $p_1, \dots, p_n$ of weights $a_i, 1-\alpha_i$ at $p_i$. The space $\M_{\alpha}$ is smooth of complex dimension $2m-6$, where
$$ m = m(\alpha) = | \{ \alpha_i \mid \alpha_i \in (0, \tfrac{1}{2}) \} |.$$
(When $m < 3$, we have $\M_{\alpha}=\emptyset.$) Boden and Yokogawa \cite{BodenYokogawa} showed that the spaces $\M_{\alpha}$ are connected and simply connected, and computed their Poincar\'e polynomials (which only depend on $m$). In particular, the Euler characteristic of $\M_{\alpha}$ is $(m-1)(m-2)2^{m-4}$.
\end{example}

\begin{example}
\label{ex:T3}
For an example where the representation variety $\Rep(Y)$ is singular, take the three-torus $T^3$, with $\pi_1(T^3)=\Z^3$. One can check that $\Rep(Y)$ has complex dimension $5$, whereas the Zariski tangent space to $\Rep(Y)$ at the trivial representation is $9$-dimensional: $Z^1(T^3; \g) \cong H^1(T^3; \g) \cong \g^3$.
\end{example}

\begin{example}
An example of a three-manifold $Y$ where the character scheme $\Char(Y)$ is non-reduced, based on  \cite[equation 2.10.4, p.43]{LubotzkyMagid}, was given on p.27 of the version \url{arXiv:1303.2347v2} of \cite{KapovichMillson}. (However, it does not appear in the published version.) The manifold in question is a Seifert fibered space over the orbifold $S^2(3,3,3)$, i.e. over the sphere with three cone points of order $3$.
\end{example}

\begin{remark}
Kapovich and Millson \cite{KapovichMillson} proved universality results for representation schemes and character schemes of three-manifolds, which show that their singularities can be ``arbitrarily complicated''. Specifically, let $Z \subset \C^N$ be an affine scheme over ${\mathbb Q}$, and $x \in Z$ a rational point. Then there exists a  natural number $k$ and a closed (non-orientable) $3$-dimensional manifold $Y$ with a representation $\rho: \pi_1(Y) \to \sl$ so that there are isomorphisms of analytic germs 
$$ (R(Y), \rho) \cong (Z \times \C^{3k+3}, x \times 0)$$
and
$$ (X(Y), [\rho]) \cong (Z \times \C^{3k}, x \times 0).$$
\end{remark}

\subsection{The case of surfaces}
\label{sec:surfaces}
Let $\Sigma$ be a closed oriented surface of genus $g \geq 2$, and let $\Gamma = \pi_1(\Sigma)$. We review a few facts about the character variety of $\Gamma$, following Goldman \cite{Goldman, GoldmanSL2C} and Hitchin \cite{Hitchin}.

A representation $\rho: \pi_1(\Sigma) \to G$ is regular if and only if it is non-Abelian. The character scheme $\Chars(\Sigma)= \Chars(\Gamma)$ is reduced, of complex dimension $6g-6$. Concretely, in terms of the images $A_i, B_i$ of the standard generators of $\pi_1(\Sigma)$, we can write the character variety as
\begin{equation}
\label{eq:CS}
 \Char(\Sigma) =\bigl \{ (A_1, B_1, \dots, A_g, B_g) \in G^{2g}  \mid \prod_{i=1}^g [A_i, B_i]=1\}\sslash \Gad.\end{equation}

The singular locus of $\Char(\Sigma)$ consists exactly of the classes of reducible representations, and is of complex dimension $2g$. The irreducible locus $\CharIrr(\Sigma)$ is a smooth complex manifold; we denote by $J$ its complex structure (coming from the complex structure on $G=\sl$). More interestingly, $\CharIrr(\Sigma)$ admits a natural complex symplectic structure, invariant under the action of the mapping class group. Explicitly, if we identify the tangent space to $\CharIrr(\Sigma)$ at some $[\rho]$ with $H^1(\Sigma; \Ad \rho)$, the complex symplectic form is the pairing
\begin{equation}
\label{eq:omegac}
\omegac: H^1(\Sigma; \Ad \rho) \times H^1(\Sigma; \Ad \rho) \to H^2(\Sigma; \C) \cong \C,
\end{equation}
which combines the cup product with the non-degenerate bilinear form $(x, y) \to \tr(xy)$ on $\g$ (which is $1/4$ of the Killing form). Alternatively, we can identify the points $[\rho]\in\Char(\Sigma)$ with flat $\sl$ connections $A_{\rho}$ on $\Sigma$ up to gauge, and $H^1(\Sigma; \Ad \rho)$ with deRham cohomology with local coefficients, 
$$H^1_{A_\rho} (\Sigma; \g) = \ker(d_{A_{\rho}}: \Omega^1(\Sigma; \g) \to \Omega^2(\Sigma; \g))/\im(d_{A_{\rho}}: \Omega^0(\Sigma; \g) \to \Omega^1(\Sigma; \g)).$$ We then have
$$\omegac(a, b) = \int_\Sigma \tr(a \wedge b),$$
where $a, b \in \Omega^1(\Sigma; \g)$ are $d_{A_\rho}$-closed forms.

Let us now equip $\Sigma$ with a Riemannian metric. Its conformal class determines a complex structure $j$. By the work of Hitchin \cite{Hitchin}, we can identify $\CharIrr(\Sigma)$ with the moduli space of stable Higgs bundles on $(\Sigma, j)$ with trivial determinant, and thus give it the structure of a hyperk\"ahler manifold. In Hitchin's notation, we now have three complex structures $I, J$ and $K=IJ$, where $I$ comes from the moduli space of Higgs bundles, and $J$ is the previous structure on $\CharIrr(\Sigma)$. We also have three symplectic forms $\omega_1$, $\omega_2$ and $\omega_3$ (in Hitchin's notation), where
$$ \omegac = -\omega_1 + i \omega_3.$$

\begin{remark}
\label{rem:exactness}
It is worth noting that $\omega_2$ and $\omega_3$ are exact forms, whereas $\omega_1$ is not; cf. \cite[p.109]{Hitchin} or \cite[Section 4.1]{KapustinWitten}.
\end{remark}

There is also a variant of the character variety that is smooth. Let us choose a basepoint $w \in \Sigma$ and a small disk neighborhood $D$ of $w$, whose boundary $\gamma = \del D$ is a loop around $w$. Then, instead of representations $\rho: \pi_1(\Sigma) \to G$, we can consider {\em twisted representations}, i.e., homomorphisms $\rho: \pi_1(\Sigma \setminus \{w\}) \to G$ with $\rho(\gamma) = -I$. Any such $\rho$ has trivial stabilizer, and is irreducible (it does not preserve any line in $\C^2$). Note also that $\Ad \rho$ is still well-defined as a representation of $\pi_1(\Sigma)$ on $\g$, because conjugation by $-I$ is the identity.

 We denote by $\Reptw(\Sigma)$ the space of twisted representations, and by $\Xtw(\Sigma)$ the {\em twisted character variety}
$$ \Xtw(\Sigma) :=  \Reptw(\Sigma)/\Gad.$$
In terms of the images $A_i, B_i$ of the standard generators of $\pi_1(\Sigma \setminus \{w\})$, we have
\begin{equation}
\label{eq:XtwAB}
 \Xtw(\Sigma) =\bigl \{ (A_1, B_1, \dots, A_g, B_g) \in G^{2g}  \mid \prod_{i=1}^g [A_i, B_i]=-1\}/\Gad.
 \end{equation}

The spaces $\Reptw(\Sigma)$ and $\Xtw(\Sigma)$ are smooth complex manifolds (and the corresponding schemes are reduced). The twisted character variety has complex dimension $6g-6$, and its tangent bundle at some $[\rho]$ is still identified with $H^1(\Sigma; \Ad \rho)$. We can equip $\Xtw$ with a complex symplectic form $\omegac = -\omega_1 + i\omega_3$, as before. In terms of gauge theory, twisted representations correspond to central curvature (i.e., projectively flat) connections in a rank two bundle of odd degree on $\Sigma$. 

After choosing a conformal structure on $\Sigma$, we can identify $\Xtw(\Sigma)$ with a moduli space of Higgs bundles of odd degree and fixed determinant, cf. \cite{Hitchin}. This gives a hyperk\"ahler structure on $\Xtw(\Sigma)$, with complex structures $I, J, K$ and symplectic forms $\omega_1, \omega_2, \omega_3$. They satisfy properties similar to those of the respective objects on $\CharIrr(\Sigma)$.

Let us end with some remarks about the case when the surface $\Sigma$ is of genus $g=1$. Then, all representations $\rho: \pi_1(\Sigma) \to G$ are reducible. The character variety $\Char(\Sigma)$ is the quotient of $\C^* \times \C^*$ by an involution, and $\CharIrr(\Sigma) = \emptyset.$ On the other hand, the twisted character variety $\Xtw(\Sigma)$ is still smooth, consisting of a single point. Indeed,
\begin{equation}
\label{eq:pair}
 A = 
\begin{pmatrix} 
i & 0 \\ 
0 & -i 
\end{pmatrix},  \ \
B = \begin{pmatrix} 
0 & 1 \\ 
-1 & 0 \end{pmatrix} 
\end{equation}
are (up to conjugation) the only pair of anti-commuting matrices in $\sl$.

\section{Lagrangians from Heegaard splittings}
\label{sec:Motivation}
As mentioned in the Introduction, the Atiyah-Floer conjecture \cite{AtiyahFloer} asserts that the  $\su$ instanton homology of a three-manifold can be constructed as Lagrangian Floer homology, for two Lagrangians inside the moduli space of flat $\su$ connections of a Heegaard surface. In this section we pursue a complex version of this construction, with $\su$ replaced by $\sl$.

The Lagrangians constructed below are examples of $(A,B,A)$ branes, in the sense that they are of type A (Lagrangian) for the complex structures $I$ and $K$ (more precisely, for the forms $\omega_1$ and $\omega_3$), and of type $B$ (complex) for the complex structure $K$. These Lagrangians have also appeared in the work of Gukov \cite{GukovSO} and that of Baraglia and Schaposnik; cf. \cite[Section 11]{BaragliaSchaposnikABA} and \cite{BaragliaSchaposnikReal}.

\subsection{Lagrangians in the character variety}
\label{sec:Lags}
Let $Y$ be a closed, connected, oriented three-manifold. Any such manifold admits a Heegaard splitting
$$ Y = U_0 \cup_{\Sigma} U_1,$$
where $\Sigma$ is a closed oriented Heegaard surface, and $U_0, U_1$ are handlebodies. We denote by $g$ the genus of $\Sigma$. 

Given a Heegaard splitting, we consider the irreducible locus of its character variety, $\CharIrr(\Sigma) \subset \Char(\Sigma)$. Note that, when $g=0$ or $1$, the group $\pi_1(\Sigma)$ is Abelian, and hence $\CharIrr(\Sigma)$ is empty. Thus, we will assume $g \geq 2$.

We equip $\CharIrr(\Sigma)$ with the complex structure $J$ and the complex symplectic form $\omegac$, as in Section~\ref{sec:surfaces}. For each handlebody $U_i, i=0,1$, let $\iota_i :\Sigma \to U_i$ be the inclusion, and $(\iota_i)_*$ the induced map on $\pi_1$. We consider the subspace 
$$L_i =\{ [\rho \circ (\iota_i)_*] \mid \rho: \pi_1(U_i) \to G \text{ irreducible} \}  \subset \CharIrr(\Sigma).$$
Equivalently, if we view $\CharIrr(\Sigma)$ as the space of irreducible flat $\sl$ connections on $\Sigma$, then $L_i$ consists of those flat connections that extend to $U_i$.

\begin{lemma}
\label{lem:idab}
$(a)$ The subspaces $L_i \subset \CharIrr(\Sigma)$ can be naturally identified with $\CharIrr(U_i)$, and their intersection $L_0 \cap L_1$ with $\CharIrr(Y)$.

$(b)$ The subspaces $L_i$ are smooth complex Lagrangians of $\CharIrr(\Sigma)$. 
\end{lemma}

\begin{proof}
$(a)$ Note that $(\iota_i)_*: \pi_1(\Sigma) \to \pi_1(U_i)$ is surjective. Consequently, two representations $\rho_1, \rho_2: \pi_1(U_i) \to G$ are the same if and only if $\rho_1 \circ (\iota_i)_*= \rho_2\circ (\iota_i)_*.$ Furthermore, a representation $\rho: \pi_1(U_i) \to G$ is reducible (fixes a line in $\C^2$) if and only if $\rho \circ i_*$ is. This gives the identifications $L_i \cong \CharIrr(U_i)$. 

The same argument can be used to identify $L_0 \cap L_1$ with $\CharIrr(Y)$. The key observation is that if $\iota: \Sigma \to Y$ denotes the inclusion, then the induced map $\iota_*$ on $\pi_1$ is surjective. This follows from the fact that $\pi_1(\Sigma)$ surjects onto $\pi_1(U_0)$ and $\pi_1(U_1)$, together with the Seifert-van Kampen theorem.

$(b)$ Let us check that $\omegac$ vanishes on $T_{[\rho]}L_i \subset T_{[\rho]}\CharIrr(Y)$. Let $A = A_{\rho}$ be the flat connection associated to $\rho$ on $U_i$. In terms of connections, the inclusion 
$T_{[\rho]}L_i \subset T_{[\rho]}\CharIrr(Y)$ corresponds to
$$ H^1_A(U_i; \g) \subset H^1_A(\Sigma; \g).$$
For $d_A$-closed forms $a, b \in \Omega^1_A(U_i; \g),$ by Stokes' theorem, we have
$$ \omegac(a, b) = \int_\Sigma \tr(a \wedge b) = \int_{U_i} d\tr(a \wedge b) = \int_{U_i} \tr(d_A a \wedge b - a \wedge d_A b) = 0.$$

Moreover, since $\pi_1(U_i)$ is the free group $F_g$ on $g$ generators, the spaces $L_i$ are diffeomorphic to the varieties $\CharIrr(F_g)$ from Section~\ref{sec:free}. These are of complex dimension $3g-3$, which is half the dimension on $\CharIrr(\Sigma)$. We conclude that $L_i$ are Lagrangians. They are also complex submanifolds, since the complex structures come from the complex structure on $\g$.
\end{proof}

Explicitly, we can choose standard generators $a_1, \dots, a_g, b_1, \dots, b_g$ of $\pi_1(\Sigma)$, with $\prod_i [a_i, b_i]=1$, such that $b_1, \dots, b_g$ also generate $\pi_1$ of one of the handlebodies, say $U_0$. If we denote $A_i=\rho(a_i)$, $B_i=\rho(b_i)$, we have the description \eqref{eq:CS} of $\Char(\Sigma)$. In terms of that description, the Lagrangian $L_0$ corresponds to the irreducible representations $\rho$ that satisfy
$$ A_i =1, \ i=1, \dots, g.$$
The second Lagrangian $L_1$ is the image of $L_0$ under an element in the mapping class group of $\Sigma$.

\subsection{Lagrangians in the twisted character variety}
\label{sec:LagTw}
We now present a twisted version of the constructions in Section~\ref{sec:Lags}. This is inspired by the torus-summed Lagrangian Floer homology in the $\su$ case, proposed by Wehrheim and Woodward in \cite[Definition 4.4.1]{WWFloerField}; see also \cite{Horton} for a related construction.

We start with a Heegaard splitting $Y = U_0 \cup_{\Sigma} U_1$ as before. (We allow the case when the Heegaard genus $g$ is $0$ or $1$.) We pick a basepoint $z$ on $\Sigma \subset Y$, and take the connected sum of $Y$ with $T^2 \times [0,1]$, by identifying a ball $B \subset Y$ around $z$ 
with a ball $B'$ in $T^2 \times [1/4,3/4] \subset T^2 \times [0,1]$. We assume that $B$ is split by $\Sigma$ into two solid hemispheres, with common boundary a two-dimensional disk $D \subset \Sigma$. Similarly, $B'$ is split into two solid hemispheres by $T^2 \times \{1/2\}$, and $D$ is identified with the intersection $B' \cap (T^2 \times \{1/2\})$. In this fashion, we obtain a decomposition of $$Y^{\#} := Y \# (T^2 \times [0,1])$$
into two compression bodies\footnote{In three-dimensional topology, a {\em compression body} is either a handlebody or the space obtained from $S \times [0,1]$ by attaching one-handles along $S \times \{1\}$, where $S$ is a closed surface. In our case, $S=T^2$, and we attach $g$ one-handles, where $g$ is the genus of $\Sigma$.} $U_0^{\#}$ and $U_1^{\#}$, each going between $\Sigma^{\#}:=\Sigma \# T^2$ and a copy of $T^2$. We also pick a basepoint $w$ on $T^2 \cong T^2 \times \{1/2\}$ (away from the connected sum region), which becomes a basepoint on $\Sigma^{\#}$. We denote by $\ell_0$ and $\ell_1$ the intervals $\{w\} \times [0,1/2]$ and $\{w\} \times [1/2, 1]$. See Figure~\ref{fig:twisting}.

\begin {figure}
\begin {center}
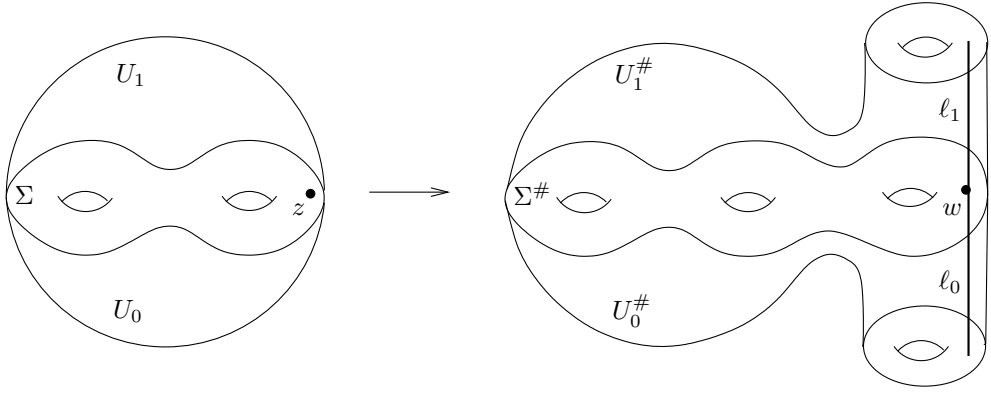
\caption {The connected sum of a Heegaard decomposition with $Y^2 \times [0,1]$.}
\label{fig:twisting}
\end {center}
\end {figure}

As explained in Section~\ref{sec:surfaces}, we can consider the twisted character variety $\Xtw(\Sigma^{\#})$, using representations $\rho: \pi_1(\Sigma \setminus \{w\}) \to G$ that take a loop around $w$ to $-I$. Inside $\Xtw(\Sigma^{\#})$ we take the subspaces $L_i^{\#}, i=0,1$, consisting of classes $[\rho]$ for representations that factor through $\pi_1(U^{\#}_i \setminus \ell_i)$. 
 
\begin{lemma}
$(a)$ The subspaces $L^{\#}_i \subset \Xtw(\Sigma^{\#})$ can be naturally identified with the representation varieties $\Rep(U_i)$, and their intersection $L^{\#}_0 \cap L^{\#}_1$ with $\Rep(Y)$.

$(b)$ The subspaces $L^{\#}_i$ are smooth complex Lagrangians of $\Xtw(\Sigma^{\#})$. 
\end{lemma}
 
 \begin{proof}
 $(a)$ Arguing as in the proof of Lemma~\ref{lem:idab}(a), we can identify $L^{\#}_i$ with a twisted character variety $\Xtw(U^{\#}_i)$, consisting of conjugacy classes of representations $\rho: \pi_1(U^{\#}_i \setminus \ell_i) \to G$ that take the value $-I$ on a loop around $\ell_i$. Since $\pi_1(U^{\#}_i \setminus \ell_i)$ is the free product of $\pi_1(U_i)$ and $\pi_1(T^2 \setminus \{w\})$,  we can write
 $$\Xtw(U^{\#}_i) = (\Rep(U_i) \times \Reptw(T^2))/\Gad.$$
Therefore, we have a fiber bundle
\begin{equation}
\label{eq:RT}
 \Rep(U_i) \hookrightarrow \Xtw(U^{\#}_i) \twoheadrightarrow \Reptw(T^2)/\Gad = \Xtw(T^2),
 \end{equation}
where the projection $\Xtw(U^{\#}_i) \twoheadrightarrow \Xtw(T^2)$ is induced by the inclusion of $T^2 \cong T^2 \times \{i\}$ into $U^{\#}_i$.

As mentioned at the end of Section~\ref{sec:surfaces}, the twisted character variety $  \Xtw(T^2)$ is a single point. Hence, the inclusion $\Rep(U_i)  \hookrightarrow \Xtw(U^{\#}_i)$ is an isomorphism. Explicitly, the inclusion takes $\rho \in \Rep(U_i)$ to the class of representation $\tilde \rho:\pi_1(U^{\#}_i \setminus \ell_i) \to G$ by mapping the generators of $\pi_1$ of the extra torus to the pair of matrices from \eqref{eq:pair}.
 
After identifying each $L^{\#}_i$ with $\Rep(U_i)$, the intersection $L^{\#}_0 \cap L^{\#}_1$ becomes the space of pairs of representations $(\rho_0, \rho_1) \in \Rep(U_1) \times \Rep(U_2)$ that have the same restriction to $\pi_1(\Sigma)$. Using the Seifert-van Kampen theorem as in the proof of Lemma~\ref{lem:idab} (a), we see that this space is  
the same as $\Rep(Y)$.
 
$(b)$ The proof of Lemma~\ref{lem:idab}(b) applies here with a slight modification: instead of flat connections we use projectively flat connections on rank two complex bundles with $c_2 \neq 0$.
 \end{proof}
 
Let us choose standard generators $a_1, \dots, a_g, b_1, \dots, b_g$ of $\pi_1(\Sigma)$, with $\prod_i [a_i, b_i]=1$, such that $b_1, \dots, b_g$ also generate $\pi_1(U_0).$ We add two more generators $a_{g+1}, b_{g+1}$ for the new torus $T^2$, and we obtain a generating set for $\pi_1(\Sigma^{\#})$. If we denote $A_i=\rho(a_i)$, $B_i=\rho(b_i)$, recall from \eqref{eq:XtwAB} that we can write
$$ \Xtw(\Sigma^{\#}) =\bigl \{ (A_1, B_1, \dots, A_g, B_g, A_{g+1}, B_{g+1}) \in G^{2g+2}  \mid \prod_{i=1}^{g+1} [A_i, B_i]=-1\}/\Gad.
 $$
Then, the Lagrangian $L_0$ is given by the equations
$$ A_i =1, \ i=1, \dots, g,$$
and $L_1$ is the image of $L_0$ under a mapping class group element. 

Observe that, since the Lagrangians are identified with $\Rep(U_i)$, they are diffeomorphic to products of $g$ copies of $G \cong S^3 \times \R^3$.

\subsection{Conditions on intersections} Let us recall the definition of clean intersections.

\begin{definition}
\label{def:clean}
Let $M$ be a smooth manifold and $L_0, L_1 \subset M$ smooth submanifolds. We say that $L_0$ and $L_1$ intersect {\em cleanly} at a point $x \in L_0 \cap L_1$ if there is a neighborhood $U$ of $x$ in $M$ such that $L_0 \cap L_1 \cap U$ is a smooth submanifold $Q$, and we have
$$ T_x Q = T_x L_0 \cap T_x L_1 \subset T_x M.$$
Furthermore, we say that $L_0$ and $L_1$ intersect cleanly if they do so at every $x \in L_0 \cap L_1$.
\end{definition}

In particular, transverse intersections are clean.

Let $Y = U_0 \cap_{\Sigma} U_1$ be a Heegaard splitting of a three-manifold. For the Lagrangians constructed in Sections~\ref{sec:Lags} and \ref{sec:LagTw}, we have the following criteria for clean and transverse intersections.

\begin{lemma}
\label{lem:cleanChar}
The Lagrangians $L_0, L_1 \subset \CharIrr(\Sigma)$ intersect cleanly at a point $[\rho] \in \CharIrr(Y)$ if and only if the representation $\rho$ is regular, i.e. $[\rho]$ is a regular point of the character scheme $\CharsIrr(Y)$ (cf. Lemma~\ref{lem:reg}).
\end{lemma}

\begin{proof}
For the Heegaard splitting $Y = U_0 \cup_{\Sigma} U_1$, the Mayer-Vietoris sequence with local coefficients reads
$$ \dots \to H^0(\Sigma; \Ad\rho) \to H^1(Y; \Ad \rho) \to H^1(U_0; \Ad \rho) \oplus H^1(U_1; \Ad \rho) \to H^1(\Sigma; \Ad \rho) \to \dots$$
Because $\rho$ is irreducible, we have 
$$H^0(\Sigma; \Ad\rho)=\{a \in \g \mid [\rho(x), a]=0,\ \forall x \in \pi_1(\Sigma)\}=0.$$
Thus, we can identify $H^1(Y; \Ad \rho) $ with the intersection
$$ H^1(U_0; \Ad \rho) \cap H^1(U_1; \Ad \rho) \subset H^1(\Sigma; \Ad \rho).$$
Since the character schemes $\CharsIrr(U_0)$, $\CharsIrr(U_1)$ and $\CharsIrr(\Sigma)$ consist of only regular representations, they are smooth (by Lemma~\ref{lem:reg}), and the tangent bundles to the corresponding varieties $L_0 = \CharIrr(U_0)$, $L_1=\CharIrr(U_1)$ and $M=\CharIrr(\Sigma)$ at $[\rho]$ are $H^1(U_0; \Ad \rho)$, $H^1(U_1; \Ad \rho)$, and $H^1(\Sigma; \Ad \rho)$. Moreover, by Proposition~\ref{prop:TC} (b), the tangent space to the scheme $\CharsIrr(Y)$ is $H^1(Y; \Ad \rho)$.
Therefore, we have 
\begin{equation}
\label{eq:tro}
T_{[\rho]} \CharsIrr(Y)= T_{[\rho]}L_0 \cap T_{[\rho]}L_1 \subset T_{[\rho]}M.
\end{equation}

Now, if $L_0$ and $L_1$ intersect cleanly at $[\rho]$ (along a submanifold $Q$, in a neighborhood of $[\rho]$), then \eqref{eq:tro} implies that $T_{[\rho]} \CharsIrr(Y)= T_{[\rho]}Q$, so $[\rho]$ is a regular point of $\CharsIrr(Y)$. Conversely, if $[\rho]$ is a regular point, then once again locally the intersection is a smooth submanifold $Q$, with $T_{[\rho]} \CharsIrr(Y)= T_{[\rho]}Q$. In view of \eqref{eq:tro}, we conclude that the intersection is clean.
\end{proof}

\begin{corollary}
\label{cor:transChar}
The Lagrangians $L_0, L_1 \subset \CharIrr(\Sigma)$ intersect transversely at a point $[\rho] \in \CharIrr(Y)$ if and only if $H^1(Y; \Ad \rho) = 0$.
\end{corollary}

\begin{proof}
By part (b) of Proposition~\ref{prop:TC}, since $\rho$ is irreducible, we have $T_{[\rho]} \CharsIrr(Y) \cong H^1(Y; \Ad \rho)$. The conclusion follows from this and the relation \eqref{eq:tro}.
\end{proof}

\begin{lemma}
\label{lem:cleanRep}
The Lagrangians $L^{\#}_0, L^{\#}_1 \subset \Xtw(\Sigma^{\#})$ intersect cleanly at a point $\rho \in \Rep(Y)$ if and only if $\rho$ is regular, i.e., $\rho$ is a regular point of the representation scheme $\Rs(Y)$ (cf. Definition~\ref{def:redreg}).
\end{lemma}

\begin{proof}
The proof is similar to that of Lemma~\ref{lem:cleanChar}, with spaces of 1-cocycles instead of first cohomology groups.

By the Seifert-van Kampen theorem, we have $\pi_1(Y) = \pi_1(U_0) *_{\pi_1(\Sigma)} \pi_1(U_1)$. This time, in view of the definition \eqref{eq:1cocycle} of 1-cocycles, we can directly identify $Z^1(\pi_1(Y); \Ad \rho)$ with the intersection
$$ Z^1(\pi_1(U_0); \Ad \rho) \cap Z^1(\pi_1(U_1); \Ad \rho) \subset Z^1(\pi_1(\Sigma); \Ad \rho).$$
Since the schemes $\Rs(U_0)$, $\Rs(U_1)$ and $\Rs(\Sigma)$ are reduced, the tangent spaces to the corresponding varieties $L_0^{\#} = \Rep(U_0)$, $L_1^{\#}=\Rep(U_1)$ and $\Rep(\Sigma)$ at $\rho$ are the spaces of $1$-cocycles. Further, the tangent space to the scheme $\Rs(Y)$ is $Z^1(\pi_1(Y); \Ad \rho)$. Therefore, we have 
\begin{equation}
\label{eq:Tro0}
T_{\rho}\Rs(Y)= T_{\rho}L^{\#}_0 \cap T_{\rho}L^{\#}_1 \subset T_{\rho}\Rep(\Sigma).
\end{equation}

We also have an inclusion 
$$\Rep(\Sigma) \hookrightarrow M^{\#} =  \Xtw(\Sigma^{\#}), \ \ \rho \mapsto [\tilde \rho],$$
where $\tilde \rho$ acts as $\rho$ on $\pi_1(\Sigma) \subset \pi_1(\Sigma^{\#})$, and takes the generators of the new torus to the anti-commuting matrices from \eqref{eq:pair}. At the level of tangent spaces, we get an inclusion
$$ T_{\rho}\Rep(\Sigma) = Z^1(\pi_1(\Sigma); \Ad \rho) \hookrightarrow T_{\rho} M^{\#} =H^1(\pi_1(\Sigma_1^{\#}); \Ad \tilde\rho),$$
where we identified $\rho$ with its image $[\tilde \rho]$.

Now, instead of \eqref{eq:Tro0}, let us write
\begin{equation}
\label{eq:Tro}
T_{\rho}\Rs(Y)= T_{\rho}L^{\#}_0 \cap T_{\rho}L^{\#}_1 \subset T_{\rho}M^{\#}.
\end{equation}
 
 If $L^{\#}_0$ and $L^{\#}_1$ intersect cleanly at $\rho$, along a submanifold $Q$, then by \eqref{eq:Tro} we have $T_{\rho} \Rs(Y)= T_{\rho}Q$, so $\rho$ is a regular point of $\Rs(Y)$. Conversely, if $\rho$ is regular, then locally the intersection is a smooth submanifold $Q$, with $T_{\rho} \Rs(Y)= T_{\rho}Q$. Using \eqref{eq:tro}, we get that the intersection is clean.
\end{proof}

\begin{corollary}
\label{cor:transRep}
The Lagrangians $L^{\#}_0, L^{\#}_1 \subset \Xtw(\Sigma^{\#})$ intersect transversely at a point $\rho \in \Rep(Y)$ if and only if $Z^1(\pi_1(Y); \Ad \rho) = 0$.
\end{corollary}

\begin{proof}
Use \eqref{eq:Tro} and the identification of $T_{\rho}\Rs(Y)$ with $Z^1(\pi_1(Y); \Ad \rho)$.
\end{proof}

\subsection{Floer homology for complex Lagrangians}
\label{sec:ComplexLag}
Let us investigate the possibility of defining Lagrangian Floer homology with the spaces constructed in Section~\ref{sec:Lags} and \ref{sec:LagTw}. (Such a construction has been explored in the physics literature, for example in \cite{GukovSO}.) We refer to \cite{FloerLagrangian, FOOO, SeidelBook, OhBook, AurouxSurvey} for references on Lagrangian Floer homology.

With regard to $L_0, L_1 \subset \CharIrr(\Sigma)$, note that both $\CharIrr(\Sigma)$ and the Lagrangians are non-compact, and in fact not even complete as metric spaces (with respect to the hyperk\"ahler metric mentioned in Section~\ref{sec:surfaces}). Thus, holomorphic strips may limit to strips that go through the reducible locus, where the character variety $\Char(\Sigma)$ is singular. Defining Floer homology in such a situation is problematic.

The situation is more hopeful for the Lagrangians $L^{\#}_0, L^{\#}_1 \subset \Xtw(\Sigma^{\#})$. They are still non-compact, but are complete with respect to the hyperk\"ahler metric, and we can try to understand their behavior at infinity.

We should also decide what symplectic form to use on the manifold $M=\Xtw(\Sigma^{\#})$. Recall from Section~\ref{sec:surfaces} that $\omegac = -\omega_1 + i \omega_3$. We can try  $\omega_1$, $\omega_3$, or a combination of these. 

\begin{remark}
The intuition behind the Atiyah-Floer conjecture is that, as we stretch the three-manifold $Y$ along a Heegaard surface $\Sigma$, the ASD Yang-Mills equations on $\R \times Y$ become the Cauchy-Riemann equations for strips in the moduli space of flat $\su$ connections on $\Sigma$. In the $\sl$ case, on $\R \times Y$ we can consider the Kapustin-Witten equations \cite{KapustinWitten} for various parameters $t \in \R$. In particular, at $t=0$ we find the $\sl$ ASD equations, and at $t=1$ we see the equations considered in Witten's work on Khovanov homology \cite{FiveBranes}. When we stretch $Y$ along $\Sigma$, we get the Cauchy-Riemann equations in $\Char(\Sigma)$, with respect to the complex structure $I$ for $t=0$, and with respect to $K$ for $t=1$; see \cite[Section 4]{KapustinWitten}. The same goes for $\Xtw(\Sigma)$ if we do a twisted version. Note that, under the hyperk\"ahler metric, the complex structure $I$ corresponds to $\omega_1$, and $K$ to $\omega_3$.

Another option is to consider the Vafa-Witten equations on $\R \times Y$ \cite{VafaWitten}. When we stretch along $\Sigma$, we obtain once again the Cauchy-Riemann equations for the complex structure $I$. See \cite[Section 4.2]{Haydys} or \cite[Section 8]{Mares}.
\end{remark}

Observe that since $M$ is hyperk\"ahler, it is Calabi-Yau ($c_1=0$). If we work with $\omega_1$ (which is not an exact form; see Remark~\ref{rem:exactness}), we expect sphere bubbles to appear, and they would not be controlled by their first Chern class. This makes constructing Lagrangian Floer homology more difficult. 

It seems better to use $\omega_3$, which is exact. Since the Lagrangians $L_i^{\#}$ are diffeomorphic to products of $G$, they satisfy $H^1(L_i^{\#}; \Z)=0$, so are automatically exact. This precludes the existence of disks and sphere bubbles. Further, since $H^1(L_i^{\#}; \Z/2)=H^2(L_i^{\#}; \Z/2)=0$, the Lagrangians have unique spin structures, and these can be used to orient the moduli space of holomorphic disks. Also, the fact that $c_1(M)=0$ implies that $M$ admits a complex volume form; a choice of a homotopy class of such volume forms induces a $\Z$-grading on the Floer homology. In fact, the  hyperk\"ahler structure determines a canonical holomorphic volume form; hence, if the Floer homology is well-defined, it admits a canonical $\Z$-grading. 

It still remains to deal with the non-compactness. To ensure that at least the intersection $L_0^{\#} \cap L_1^{\#}$ is compact, one needs to perturb one Lagrangian near infinity, by a suitable Hamiltonian isotopy. This should lead to an ``infinitesimally wrapped'' Lagrangian Floer homology, provided that (after perturbation) holomorphic strips do not escape to infinity. In particular, we need the following tameness condition on Lagrangians introduced by Sikorav.

\begin{conjecture}
\label{conj:tame}
The Lagrangians $L^{\#}_0, L^{\#}_1 \subset \Xtw(\Sigma^{\#})$ are tame, in the sense of \cite[Definition 4.7.1]{Sikorav}.
\end{conjecture}

Conjecture~\ref{conj:tame} would imply compactness for the moduli spaces of holomorphic disks with boundary on either Lagrangian. We expect that a similar tameness condition can be formulated for the pair $(L_0, L_1)$, to ensure compactness for the moduli spaces of strips.  If these conditions are all satisfied, then the Lagrangian Floer homology $\HF_*(L^{\#}_0, L^{\#}_1)$ would be well-defined. We also expect $\HF_*(L^{\#}_0, L^{\#}_1)$ to be an invariant of $Y$. A potential strategy for proving invariance would be to use the theory of Lagrangian correspondences and pseudo-holomorphic quilts developed by Wehrheim and Woodward; see \cite{WWFunctoriality}, \cite{WWFloerField}. 
 
Moreover, since $L^{\#}_0, L^{\#}_1$ are complex Lagrangians, there should be no non-trivial pseudo-holomorphic strips between then. Indeed:
\begin{itemize}
\item If two $J$-complex Lagrangians in a hyperk\"ahler manifold $(M, I, J, K, g)$ intersect transversely, then the relative Maslov grading between any two intersection points is always zero. Indeed, the relative grading is the index of an operator $L$, the linearization of the Cauchy-Riemann operator (defined from the complex structure $I$ or $K$). One can check that the operators $J^{-1}LJ$ and $-L^*$ differ by a compact operator, which implies that $\ind(L) = \ind(L^*) = 0.$ This is an analogue of the fact that, in finite dimensions, the Morse index of the real part of a holomorphic function is zero (because the signature of a complex symmetric bilinear form is zero). Since the relative grading is zero, for generic almost complex structures, the moduli space of pseudo-holomorphic curves is empty;

\item Even if the two $J$-complex Lagrangians do not intersect transversely, for a generic value of $\theta \in \R$, if we consider the complex structure $K(\theta) =\cos(\theta) K +  \sin(\theta) I$, then there are no $K(\theta)$-holomorphic strips with boundary on the Lagrangians; cf. \cite{SolomonVerbitsky}. Note that $K(\theta)$ is $\omega_3$-tame for $\theta$ close to $0$.
\end{itemize}

The above results suggest that the Lagrangian Floer homology of complex Lagrangians may have a simpler algebraic interpretation. Indeed, in \cite[Remark 6.15]{BBDJS}, the authors describe an analogy between Lagrangian Floer homology and a sheaf-theoretic construction. In the following sections we will follow their suggestion and construct three-manifold invariants using sheaf theory instead of symplectic geometry.

\section{Sheaves of vanishing cycles and complex Lagrangians}
\label{sec:sheaves}
In this section we review some facts about complex symplectic manifolds, perverse sheaves, vanishing cycles, and then present Bussi's construction from \cite{Bussi}.

\subsection{Complex symplectic geometry}
We start with a few basic definitions and results; some of them also appear in \cite[Section 1.3]{Bussi}.

\begin{definition}
A {\em complex symplectic manifold} $(M, \omega)$ is a complex manifold equipped with a closed non-degenerate holomorphic two-form $\omega$. If $M$ has complex dimension $2n$, an $n$-dimensional complex submanifold $L \subset M$ is called {\em complex Lagrangian} if $\omega|_L=0$.
\end{definition}

The standard example of a complex symplectic manifold is $T^*\C^n$, with the canonical symplectic form $\omega_{\can}$. We have a complex Darboux theorem, whose proof is the same as in the real case.
\begin{theorem}
\label{thm:Darboux}
Let $(M, \omega)$ be a complex symplectic manifold, and pick $p \in M$. Then, there exist a neighborhood $S$ of $p$ and an isomorphism (i.e. biholomorphic symplectomorphism) $h: (S, \omega) \to (T^*N, \omega_{\can})$, for an open set $N \subseteq \C^n$.
\end{theorem}

There is also a complex Lagrangian neighborhood theorem:
\begin{theorem}
\label{thm:CxLagNbhd}
Let $(M, \omega)$ be a complex symplectic manifold, and $Q \subset M$ a complex Lagrangian. For any $p \in Q$, there exist a neighborhood $S$ of $p$ in $M$ and an isomorphism $h: (S, \omega) \to (T^*N, \omega_{\can})$, for an open set $N \subseteq \C^n$, such that $h(Q \cap S) = N$, the zero section in $T^*N$.
\end{theorem}

Note that, unlike in the real case, Theorem~\ref{thm:CxLagNbhd} does not describe a neighborhood of the whole Lagrangian $Q$. In the complex setting, a neighborhood of $Q$ may not be isomorphic to $T^*Q$. This is related to the fact that complex manifolds may have nontrivial moduli.

We now discuss polarizations, starting with the linear case.
\begin{definition}
A {\em polarization} of a complex symplectic vector space $(V, \omega)$ is a linear projection $\pi: V \to V/L$, determined by the choice of a complex Lagrangian subspace $L \subset V$. 
\end{definition}

Given a polarization, we can choose another Lagrangian subspace $Q \subset V$, transverse to $L$, identify $V/L$ with $Q$ and get a decomposition $V=Q \oplus L$, as well as an isomorphism $L \cong Q^*$ induced by the symplectic form. Overall, we get a decomposition 
\begin{equation}
\label{eq:polar}
V = Q \oplus Q^*.
\end{equation}  Observe that, given $L$ and $Q$, any other $Q'$  transverse to $L$ is described as the graph of a linear function $f: Q \to L$, which is symmetric iff $Q'$ is Lagrangian. Therefore, given the polarization $L$, the space of possible $Q$ is the space of symmetric matrices, which is  contractible. Thus, we sometimes think of polarizations (informally) as decompositions \eqref{eq:polar}.

\begin{definition}
\label{def:polM}
A {\em polarization} of a complex symplectic manifold $(M, \omega)$ is a holomorphic Lagrangian fibration $\pi: S \to Q$, where $S \subset M$ is open, and $Q$ is a complex manifold.
\end{definition}

By slightly refining the proof of Darboux's theorem, we obtain the following results.
\begin{theorem}
\label{thm:pol1}
(a) Let $(M, \omega)$ be a complex symplectic manifold. Suppose we are given $p \in M$ and a polarization $\sigma: T_pM \to T_pM/L_p$. Then,  there exist a neighborhood $S$ of $p$ in $M$, an open subset $N \subset \C^n$, an isomorphism $h: S \to T^*N$ as in Theorem~\ref{thm:Darboux}, and a polarization $\pi: S \to Q$ such that $(d\pi)_p: T_pM \to T_pQ$ is the linear polarization $\sigma$; that is, $\ker(d\pi)_p=L_p$.

(b) Let $(M, \omega)$ be a complex symplectic manifold. Suppose we are given $p \in M$, a polarization $\sigma: T_pM \to T_pM/L_p$, and also a complex Lagrangian submanifold $Q \subset M$ through $p$, such that $T_pQ$ intersects $L_p$ transversely. Then, we can find a neighborhood $S$ of $p$ in $M$, an open subset $N \subset \C^n$, an isomorphism $h: S \to T^*N$ as in Theorem~\ref{thm:CxLagNbhd}, with $h(Q\cap S) = N$, and a polarization $\pi: S \to Q$ such that $(d\pi)_p=\sigma$.
\end{theorem}

A {\em complex symplectic bundle} $E$ over a space $X$ is a complex vector bundle over $X$ equipped with continuously varying linear symplectic forms in the fibers. A {\em holomorphic symplectic bundle} $\E$ over a complex manifold $M$ is a holomorphic bundle over $M$ equipped with linear symplectic forms in the fibers, which produce a holomorphic section of $(\E \otimes \E)^*$. 

We can extend the notion of polarization to these kinds of bundles. 

\begin{definition}
Let $M$ be a complex manifold, and $E \to M$ a complex symplectic vector bundle.
A {\em polarization} in $E$ is a bundle map (projection) $\pi: E \to E/L$, given by the choice of a complex Lagrangian subbundle $L \subset E$.

Furthermore, if $\E$ is a holomorphic symplectic bundle, and $\L$ is a holomorphic Lagrangian subbundle, we say that $\pi$ is a {\em holomorphic polarization}.
\end{definition}

If a complex symplectic vector bundle $E$ has a polarization $\pi: E \to E/L$, we can find a Lagrangian subbundle $Q \subset E$ transverse to $E$ (using the contractibility of the space of such local choices). This gives a decomposition
\begin{equation}
\label{eq:polarbundles}
E = Q \oplus Q^*.
\end{equation}

For holomorphic bundles equipped with a holomorphic polarization $\pi: \E \to \E/\L$, we may not always find another holomorphic Lagrangian subbundle $\Q \subset \E$ transverse to $\L$, to identify $\E/\L$ with $\Q$. Thus, we do not automatically obtain a decomposition of the form  \eqref{eq:polarbundles}. 

\begin{example}
Let $M = \cp^1$, and $\E = \O \oplus \O$ with the standard complex symplectic structure on the fibers (such that the two copies of $\O$ are dual to each other). Let also $\L=\O(-1) \subset \E$ be the tautological bundle, viewed via the usual inclusion of lines in $\C^2$.  Then $\L$ gives a holomorphic polarization, but the quotient $\E/\L$ is isomorphic to $\O(1)$, which cannot be a subbundle of $\E$.
\end{example}

Finally, we mention a few well-known facts about spin structures. 
\begin{fact}
\label{fact:spin}
$(a)$ A complex vector bundle $E$ admits a spin structure if and only if $w_2(E)=0$ or, equivalently, the mod $2$ reduction of $c_1(E)$ vanishes. 

$(b)$ If they exist, spin structures on $E$ are in (non-canonical) bijection to the elements of $H^1(M; \Z/2).$ 

$(c)$ If $\E$ is holomorphic vector bundle, then a spin structure on $\E$ is the same as the data of a (holomorphic) square root of the determinant line bundle $\det(\E).$

$(d)$ If $E$ is a complex symplectic vector bundle, then the symplectic form gives rise to a trivialization of $\det(E)$. Hence, $c_1(E)=0$, so $E$ admits a spin structure. 
\end{fact}
In particular, we will be interested in spin structures on complex manifolds $M$, i.e., on their tangent bundles. Such a spin structure is the same as the choice of a square root for the anti-canonical bundle $\det(TM)$ or, equivalently (after dualizing), of a square root $K_{M}^{1/2}$ for the canonical bundle $K_M=\det(TM)^*$.

\begin{remark}
When $L$ is a complex Lagrangian, a spin structure on $L$ is called an {\em orientation} in  \cite[Definition 1.16]{Bussi}. To prevent confusion with actual orientations, we will not use that terminology in this paper.
\end{remark}

\subsection{Perverse sheaves and vanishing cycles}
We now briefly review perverse sheaves on complex analytic spaces, in the spirit of \cite[Section 1.1]{Bussi}. Almost everything that we state goes back to the original work of Be{\u\i}linson,  Bernstein, and Deligne \cite{BBD}, but given the likelihood that the reader is more comfortable with the English language, we refer instead to Dimca's book \cite{Dimca} for details.

We will work over the base ring $\Z$. Let $X$ be a complex analytic space, and $D^b_c(X)$ the derived category of (complexes of) sheaves of $\Z$-modules on $X$ with constructible cohomology. We can consider the constant sheaf $\Z_X$ (or, more generally, a local system on $X$) to be an object of $D^b_c(X)$, supported in degree zero. 

On $D^b_c(X)$ we have Grothendieck's six operations $f^*, f^{!}, Rf_*, Rf_{!}, \sRHom, \otimes^L$, as well as the Verdier duality functor $\DD_X : D^b_c(X) \to D^b_c(X)^{\operatorname{op}}$.

To an object $\Cb \in D^b_c(X)$ we can associate its hypercohomology and hypercohomology with compact support, defined by
$$ \HH^k(\Cb) = H^k(R\pi_* (\Cb)), \ \ \   \HH^k_{\cs}(\Cb) = H^k(R\pi_{!} (\Cb)),$$
where $\pi: X \to *$ is the projection to a point. In particular, for $\Cb=\Z_X$, we recover the ordinary cohomology (resp. cohomology with compact support) of $X$.

Hypercohomology and hypercohomology with compact support are related by Verdier duality:
$$ \HH^k_{\cs}(\Cb) \otimes_{\Z} \k \cong \bigl(\HH^{-k}(\DD_X(\Cb)) \otimes_{\Z} \k\bigr)^*,$$
where $\k$ is any field. Over $\Z$, we have a (non-canonical) isomorphism as in the universal coefficients theorem:
\begin{equation}
\label{eq:Vdual}
\HH^k_{\cs}(\Cb) \cong \Hom(\HH^{-k}(\DD_X(\Cb)), \Z) \oplus \Ext^1(\HH^{-k-1}(\DD_X(\Cb)), \Z).
\end{equation}

For $x \in X$, let us denote by $i_x: * \hookrightarrow X$ the inclusion of $x$. 
\begin{definition}
A {\em perverse sheaf} on $X$ is an object $\Cb \in D^b_c(X)$ such that
$$ \dim \{ x \in X \mid H^{-m}(i_x^*\Cb) \neq 0 \text{ or } H^{m}(i_x^{!}\Cb) \neq 0 \}  \leq 2m$$
for all $m \in \Z$.
\end{definition}

\begin{example}
Let $X$ be a complex manifold of complex dimension $n$, and $\L$ a $\Z$-local system on $X$. Then $\L[n]$ is a perverse sheaf on $X$.
\end{example}

Let $\Perv(X)$ be the full subcategory of $D^b_c(X)$ consisting of perverse sheaves. Then $\Perv(X)$ is an Abelian category (unlike  $D^b_c(X)$, which is only triangulated). Another way in which perverse sheaves behave more like sheaves rather than complexes of sheaves (elements of $D^b_c(X)$) is that they satisfy the following descent properties.

\begin{theorem}
\label{thm:descent}
Let $\{U_i\}_{i \in I}$ be an analytic open cover for $X$. 

$(a)$ Suppose $\Pb, \Qb$ are perverse sheaves on $X$, and for each $i$ we have a morphism $\alpha_i : \Pb|_{U_i} \to \Qb|_{U_i}$ in $\Perv(U_i)$, such that $\alpha_i$ and $\alpha_j$ agree on the double overlap $U_i \cap U_j$, for all $i, j \in I$. Then, there is a unique morphism $\alpha: \Pb \to \Qb$ in $\Perv(X)$ whose restriction to each $U_i$ is $\alpha_i$.

$(b)$ Suppose for each $i \in I$ we have a perverse sheaf $\Pb_i$ on $U_i$, and we are given isomorphisms $\alpha_{ij}: \Pb_i|_{U_i \cap U_j} \to \Pb_j|_{U_i \cap U_j}$. Suppose $\alpha_{ii} = \id$ for all $i$, and that on triple overlaps $U_i \cap U_j \cap U_k$ we have $\alpha_{jk} \circ \alpha_{ij}= \alpha_{ik}$. 

Then, there exists $\Pb \in \Perv(X)$, unique up to canonical isomorphism, with isomorphisms $\beta_i : \Pb|_{U_i} \to \Pb_i$ for all $i \in I$, such that $\alpha_{ij} \circ \beta_i|_{U_i \cap U_j} = \beta_j|_{U_i \cap U_j}$ for $i, j\in I$.
\end{theorem}

Further examples of perverse sheaves come from vanishing cycles. Given a holomorphic function $f: X \to \C$, denote $X_0 = f^{-1}(0)$ and $X_* = X \setminus X_0$. Let $\rho: \widetilde{\C^*} \to \C^*$ be the universal cover of $\C^* = \C \setminus \{0\}$, and $p: \widetilde{X_*} \to X_*$ the $\Z$-cover of $X$ obtained by pulling back $\rho$ under $f$. Let $\pi:  \widetilde{X_*} \to X$ be the composition of $p$ with the inclusion of $X_*$ into $X$, and let $i: X_0 \hookrightarrow X$ be the inclusion. We then have a {\em nearby cycle functor}
$$ \psi_f: D^b_c(X) \to D^b_c(X_0), \ \ \psi_f = i^* \circ R\pi_* \circ \pi^*.$$

For each $\Cb \in D^b_c(X)$, there is a comparison morphism $\Xi(\Cb): i^* \Cb \to \psi_f(\Cb)$. We define the {\em vanishing cycle functor} $ \phi_f: D^b_c(X) \to D^b_c(X_0)$ by extending $\Xi(\Cb)$ to a distinguished triangle
$$ i^* \Cb \xrightarrow{\phantom{aa}\Xi(\Cb)\phantom{aa}} \psi_f(\Cb) \longrightarrow \phi_f(\Cb)  \longrightarrow i^* \Cb[1]$$
in $D^b_c(X_0)$. 

\begin{theorem}[cf. Theorem 5.2.21 in \cite{Dimca}]
The shifted functors $\psi^p_f := \psi_f[-1]$ and $\phi_f^p := \phi_f[-1]$ both map $\Perv(X)$ into $\Perv(X_0)$.
\end{theorem}

To make this more concrete, suppose $U$ is an open subset of the affine space $\C^n$, and $f: U \to \C$ is holomorphic. For every $x \in U_0=f^{-1}(0)$, we define the {\em Milnor fiber} $F_x$ to be the intersection of a small open ball $B_\delta(x) \subset \C^n$ (of radius $\delta$) with the fiber $f^{-1}(\epsilon)$, for $0 < \epsilon \ll \delta$. By \cite[Proposition 4.2.2]{Dimca}, we have a natural isomorphism
$$ H^k(\psi_f \Cb)_x \cong H^k(F_x, \Cb).$$
In particular, if $\Cb = \Z_U$ and $x$ is the unique critical point of $f$, then for $y \neq x$ the cohomology $ H^k(\psi_f(\Z_X))_y$ is $\Z$ in degree zero, and $0$ otherwise. At $x$ we have 
$$H^k(\psi_f\Z_U)_y \cong H^k(F_x; \Z) \cong \begin{cases} \Z & \text{if }k=0,\\
\Z^{\mu}& \text{if }k=n-1,\\
0 & \text{otherwise},
\end{cases}
$$
where $\mu$ is the Milnor number of $f$ at $x$. 

As for the vanishing cycle $\phi_f\Z_U$, its cohomology is supported at $x$, where it is given by the reduced cohomology $\tilde{H}^*(F_x; \Z)$, which is $\Z^{\mu}$ in degree $n-1$. Thus, if we consider the perverse sheaf $\Z_U[n]$, its image under $\phi_f^p$ is (up to isomorphism in $\Perv(U_0)$) the skyscraper sheaf supported at $x$ in degree zero, with stalk $\Z^{\mu}$.

\begin{example}
\label{ex:U}
When $U=\C^n$ with coordinates $x=(x_1, \dots, x_n)$ and $f$ is given by $f(x_1, \dots, x_n)=x_1^2 + \dots + x_n^2$, then the unique critical point is $x=0$. The Milnor fiber $F_x$ is diffeomorphic to $TS^{n-1}$, and the Milnor number is $\mu=1$. Therefore, $\phi_f^p(\Z_U[n])$ is the skyscraper sheaf $\Z$ at $x=0$, in degree $0$.
\end{example}

Now suppose we have a complex manifold $U$, and $f: U \to \C$ a holomorphic function. Let $X=\Crit(f)$ be the critical locus of $f$. Note that $f|_X : X \to \C$ is locally constant, so $X$ decomposes as a disjoint union of components $X_c=f^{-1}(c) \cap X$, over $c \in f(X)$. Following \cite[Definition 1.7]{Bussi}, we define the {\em perverse sheaf of vanishing cycles} of $(U, f)$ to be
\begin{equation}
\label{eq:PVb}
 \PVb_{U, f} = \bigoplus_{c \in f(X)} \phi^p_{f-c}(\Z_U[\dim U])|_{X_c}.
 \end{equation}

\begin{example}
\label{ex:PVclean}
Let $U =\C^n$ and $f(x_1, \dots, x_n)=x_{k+1}^2 + \dots + x_n^2$, for some $k$ with $0 \leq k \leq n$. Then $X = X_0 = \C^k \subset \C^n$ is the subspace with coordinates $x_1, \dots, x_k$. When $k=0$, we are in the setting of Example~\ref{ex:U} and $\PVb_{U,f}$ is the skyscraper sheaf $\Z$ over $0$. In general, $\PVb_{U,f}$ is the constant sheaf $\Z_X[k]$ over $X$.
\end{example}

\subsection{Bussi's construction}
\label{sec:Bussi}
We are now ready to review Bussi's work from \cite{Bussi}, which associates to a pair of complex Lagrangians a perverse sheaf on their intersection.

Let $(M,\omega)$ be a complex symplectic manifold, and $L_0, L_1 \subset M$ two complex Lagrangians. We assume that $L_0$ and $L_1$ are equipped with spin structures, that is, square roots $K_{L_0}^{1/2}$ and $K_{L_1}^{1/2}$. (See Fact~\ref{fact:spin} and the paragraph after it.) 

We denote by $X$ the intersection $L_0 \cap L_1$. It will be important to view $X$ not solely as a subset of $M$, but as a complex analytic space (the complex-analytic analogue of a scheme); that is, we keep track of the structure sheaf $\O_X$. In particular, $X$ may not be reduced, and we denote by $X^{\red}$ its reduced subspace (with the same underlying topological space as $X$).

\begin{definition}
\label{def:L0chart}
An {\em $L_0$-chart} on $M$ is the data $(S, P, U, f, h, i)$, where:
\begin{itemize}
\item $S \subset M$ is open; 
\item $P= S \cap X$ and $U=S \cap L_0$;
\item $f: U \to \C$ is a holomorphic function;
\item $h:S \to T^*U$ is an isomorphism that takes $U$ to the zero section, $S \cap L_1$ to the graph of $df$, and $P$ to the critical locus $\Crit(f)$;
\item $i: P \to \Crit(f) \subset U$ is the isomorphism of analytic sets induced by the inclusion $P \hookrightarrow U$.
\end{itemize}
\end{definition}

\begin{remark}
To be consistent with the convention in \cite[Section 2]{Bussi}, we will drop $S$ and $h$ from the notation, and denote the $L_0$-chart by $(P, U, f, i)$.
\end{remark}

We can construct $L_0$-charts around any $x \in X$, as follows. We start by choosing a polarization of $T_xM$ that is transverse to both $L_0$ and $L_1$. Using Theorem~\ref{thm:pol1} (b), we can extend this to a local polarization $\pi: S \to U$ (in the sense of Definition~\ref{def:polM}) that is transverse to $L_0$ and $L_1$. This gives the desired $L_0$-chart. Conversely, an $L_0$-chart gives a polarization $\pi: S \to U$, obtained by pulling back under $h$ the projection $T^*U \to U$.

Given an $L_0$-chart $(P, U, f, i)$, the polarization $\pi: S \to U$ naturally induces a local biholomorphism between $L_0$ and $L_1$, and thus a local isomorphism between their canonical bundles
$$ \Theta: K_{L_0}|_P \xrightarrow{\cong} K_{L_1}|_P.$$
We denote by 
$$ \pi_{P, U, f, i} : Q_{P, U, f, i} \to P$$
the principal $\Z_2$-bundle parametrizing local isomorphisms between the chosen square roots (spin structures)
$$ \vartheta: K^{1/2}_{L_0}|_P \xrightarrow{\cong} K^{1/2}_{L_1}|_P$$
such that $\vartheta \otimes \vartheta = \Theta.$

On the critical locus $\Crit(f)$, we have a perverse sheaf of vanishing cycles $\PVb_{U,f}$ as in \eqref{eq:PVb}. We pull it back to $X$ under the isomorphism $i$, and then twist it by tensoring it with the bundle $Q_{P, U, f, i}$. This produces a perverse sheaf over $P \subset X$, for any $L_0$-chart. Using the descent properties (Theorem~\ref{thm:descent}), Bussi shows that one can glue these perverse sheaves to obtain a well-defined object 
$$\Pb_{L_0, L_1} \in \Perv(X)$$
with the property that for any $L_0$-chart there is a natural isomorphism
\begin{equation}
\label{eq:omegaL0}
 \omega_{P,U,f, i}: \Pb_{L_0, L_1}|_P \xrightarrow{\phantom{b}\cong \phantom{b}}  i^*(\PVb_{U,f}) \otimes_{\Z_2} Q_{P, U, f, i}.
 \end{equation}

The hypercohomology $\HH^*(\Pb_{L_0, L_1})$ is a sheaf-theoretic model for the Lagrangian Floer cohomology of $L_0$ and $L_1$. 

\section{A stabilization property}
\label{sec:prop}
In this section we establish a property of the perverse sheaves $\Pb_{L_0, L_1}$ that will be useful to us when constructing the three-manifold invariants in Section~\ref{sec:Invariant}.

\begin{proposition}
\label{prop:stabL}
Let $(M', \omega)$ be a complex symplectic manifold, and $M \subset M'$ a complex symplectic  submanifold. We denote by $\Phi: M \hookrightarrow M'$ the inclusion. We are given complex Lagrangians $L_0, L_1 \subset M$ and $L_0', L_1' \subset M'$ satisfying $ L_0 \subset L_0', \ L_1 \subset L_1'$
and 
$$L_0 \cap L_1 = L_0' \cap L_1'$$
as complex analytic spaces.

Let $N:=N_{MM'} = (TM)^{\omega}$ be the sub-bundle of $TM'|_M$ which is the symplectic complement to $TM$. Suppose we have a direct sum decomposition of $N$ into holomorphic Lagrangian sub-bundles
$$ N = V_0 \oplus V_1.$$
From here we obtain a direct sum decomposition
$$ TM'|_M = TM \oplus V_0 \oplus V_1.$$
We assume that, under this decomposition, the tangent spaces to the Lagrangians are related by
$$ TL_0'|_{L_0} = TL_0 \oplus V_0|_{L_0} \oplus 0, \  \ TL_1'|_{L_1} = TL_1 \oplus 0 \oplus V_1|_{L_1}.$$
 
Further, we assume that the Lagrangians $L_0, L_1, L_0', L_1'$ come equipped with spin structures, such that, for $i=0,1$, the spin structure on $L_i'$ is the direct sum of that on $L_i$ and a given spin structure on $V_i$. Also, the spin structure on $V_1$ should be obtained from the one on $V_0$ via the natural duality isomorphism $V_1 \cong V_0^*$ induced by $\omega$.

We are also given a non-degenerate holomorphic quadratic form $q \in H^0(\Sym^2 V_0^*)$. We assume that the spin structure on $V_0$ is self-dual under the isomorphism $V_0 \cong V_0^*$ induced by $q$.

Then, we obtain a natural isomorphism of perverse sheaves on $X=L_0 \cap L_1$:
$$\Stab: \Pb_{L_0, L_1} \xrightarrow{\phantom{b}\cong\phantom{b}} \Pb_{L'_0, L'_1}.$$
\end{proposition}

\begin{proof}
The bilinear form associated to $q$ gives a holomorphic section of $\Hom(V_0, V_0^*)$. We can think of it as a bundle map $s: V_0 \to V_0^*$, which is an isomorphism in every fiber. We identify $V_0^*$ with $V_1$, and let $W \subset N=V_0 \oplus V_1$ be the graph of $s$. Then, $W$ is a holomorphic Lagrangian sub-bundle of $N$, and the linear projection $\pi_N : N \to N/W$ is a global holomorphic polarization of $N$, transverse to $V_0$ and $V_1$. 

Near every $x \in X$, choose a polarization $\pi_{M, x}$ of the tangent space $T_xM$ transverse to $T_xL_0$ and $T_xL_1$. This induces a polarization $\pi_S$ on a neighborhood $S$ of $x$ in $M$.  As described in Section~\ref{sec:Bussi}, we can find an $L_0$-chart $(P, U, f, i)$ induced by this polarization, with open neighborhoods $P \subset X$ and $U \subset L_0$ around $x$, a holomorphic function $f: U \to \C$, the inclusion $ P \hookrightarrow U$ giving rise to an isomorphism $i: P \to \Crit(f) \subset U$, and the other Lagrangian $L_1$ represented locally as the graph of $df$. We get a natural identification
\begin{equation}
\label{eq:P01}
 \Pb_{L_0, L_1}|_P \cong i^*(\PVb_{U, f}) \otimes_{\Z_2} Q_{P, U, f, i},
\end{equation}
with $Q_{P, U, f, i}$ being the principal $\Z_2$-bundle on $P$ that parameterizes square roots of the local isomorphism $\Theta: K_{L_0}|_X \to K_{L_1}|_X$. Here, $\Theta$ is induced by the polarization $\pi_S$. The sections of $Q_{P, U, f, i}$ are local isomorphisms between $K_{L_0}^{1/2}$ to $ K_{L_1}^{1/2}$.

We now combine the polarizations $\pi_{M, x}$ and $\pi_N$ to obtain a polarization $\pi_{M', x}$ for $T_xM'$, transverse to $T_xL_0'$ and $T_xL_1'$. From here we obtain a polarization $\pi_{S'}$ of  neighborhood $S' \supset S$ of $x$ in $M$, such that $\pi_{S'}$ restricts to $\pi_S$ on $S$. Next, we obtain an $L_0'$-chart $(P', U', g, j)$ induced by $\pi_{S'}$, and extending our previous chart $(P, U, f, i)$. Here, $P' \subseteq P$ is a possibly smaller neighborhood of $x$ in $X$, the Lagrangian $L_1'$ is locally the graph of $dg$, the function $g: U' \to \C$ satisfies $g|_U = f$, and $j$ is the composition of $i$ with the restriction to $\Crit(f)$ of the inclusion $\Phi: M \to M'$. We have
\begin{equation}
\label{eq:P01'}
  \Pb_{L_0', L_1'}|_{P'} \cong j^*(\PVb_{U', g}) \otimes_{\Z_2} Q_{P', U', g, j},
\end{equation}
where $Q_{P', U', g, j}$ parameterizes square roots of the local isomorphism $\Theta': K_{L'_0}|_X \to K_{L'_1}|_X$, induced by $\pi_{S'}$.  We view the sections of $Q_{P', U', g, j}$ as local isomorphisms between $K_{L'_0}^{1/2}$ and $ K_{L'_1}^{1/2}$.

We can relate $\PVb_{U', g}$ to $\PVb_{U, f}$ by applying Theorem 1.13 in \cite{Bussi}. This gives a natural identification
\begin{equation}
\label{eq:PUU}
 \PVb_{U, f} \cong \Phi|_X^*(\PVb_{U', g})  \otimes_{\Z_2} P_{\Phi}.
\end{equation}
where $P_{\Phi}$ parametrizes square roots of the local isomorphism 
$$ J_{\Phi}: K_{L_0}^{\otimes 2}|_{X^{\red}} \xrightarrow{\cong} \Phi|^*_{X^{\red}} (K_{L_0'}^{\otimes 2})$$
induced by $q$. Indeed, by construction, the quadratic form that appears in the definition of $J_{\Phi}$ in \cite[Definition 1.11]{Bussi} is the restriction of our given $q \in H^0(\Sym^2 V_0^*)$. 

Moreover, we have
$$ \Phi^* K_{L_0'} \cong K_{L_0} \otimes \det(V^*_0).$$
Thus, the sections of $J_{\Phi}$, which are locally defined maps from $K_{L_0}|_{X^{\red}}$ to $ \Phi|^*_{X^{\red}} (K_{L_0'})$, can be interpreted as local sections of $ \det(V_0^*)$ that square to $\det(q)$. 

Let us also compare the bundle $Q_{P, U, f, i}$ from \eqref{eq:P01} to the bundle $Q_{P', U', g, j}$ from \eqref{eq:P01'}. We have
$$\Phi^* K_{L_i'}^{1/2} \cong K_{L_i}^{1/2} \otimes \det(V^*_i)^{1/2}, \ i=0,1,$$
where $\det(V^*_i)^{1/2}$ are the duals of the given spin structures on $V_i$. Therefore, 
\begin{equation}
\label{eq:RPhi}
  Q_{P', U', g, j} \cong  Q_{P, U, f, i}|_{P'} \otimes_{\Z_2}  R_{\Phi},
\end{equation}
where the sections of $R_{\Phi} \to P'$ are maps $ \det(V^*_0)^{1/2} \to \det(V^*_1)^{1/2}$, whose squares are the isomorphism between $ \det(V^*_0)$ and $\det(V^*_1)$ induced by $\omega$ and $\det(q)$. The form $q$ makes an appearance because we used it to relate the polarization on $S \subset M$, which gives \eqref{eq:P01}, to the polarization on $S' \subset M'$, which gives \eqref{eq:P01'}.

We claim that we have a canonical isomorphism
\begin{equation}
\label{eq:PR}
(i^* P_{\Phi})|_{P'} \cong R_{\Phi}.
\end{equation} Indeed, recall that the spin structures on $V_0$ and $V_1$ are related by the duality isomorphism induced by $\omega$, and the spin structure on $V_0$ is self-dual via $q$. From the isomorphisms $\det(V_0)^{1/2} \cong \det(V_0^*)^{1/2} \cong  \det(V^*_1)^{1/2},$ we get an isomorphism
$$  \det(V^*_0) \cong \Hom( \det(V^*_0)^{1/2}, \det(V^*_1)^{1/2}),$$
under which the sections of $P_{\Phi}$ and $R_{\Phi}$ correspond to each other. This proves the claim.

Combining \eqref{eq:P01}, \eqref{eq:P01'}, \eqref{eq:PUU}, \eqref{eq:RPhi} and \eqref{eq:PR}, we obtain
\begin{align*}
 \Pb_{L_0, L_1}|_{P'} &\cong i^*(\PVb_{U, f})|_{P'} \otimes_{\Z_2} Q_{P, U, f, i}|_{P'},\\
&\cong i^*(\Phi|_X^*(\PVb_{U', g})  \otimes_{\Z_2} P_{\Phi} )|_{P'} \otimes_{\Z_2} Q_{P, U, f, i}|_{P'},\\
&\cong j^*(\PVb_{U', g})   \otimes_{\Z_2} R_{\Phi}  \otimes_{\Z_2} Q_{P, U, f, i}|_{P'},\\
&   \cong j^*(\PVb_{U', g}) \otimes_{\Z_2} Q_{P', U', g, j} \\
& \cong   \Pb_{L_0', L_1'}|_{P'}.
\end{align*}

This is a local isomorphism between $\Pb_{L_0, L_1}$ and $\Pb_{L'_0, L'_1}$, defined on the open set $P'$. We can construct such isomorphisms canonically, near every $x \in X$, so that they agree on double overlaps. Using the descent property of perverse sheaves, Theorem~\ref{thm:descent} (a), we glue together the isomorphisms to obtain the desired global isomorphism.
\end{proof}

\section{Clean intersections} 
\label{sec:Clean}
In this section we study Bussi's perverse sheaf of vanishing cycles in the case where the Lagrangians intersect cleanly (in the sense of Definition~\ref{def:clean}).

We start by describing the local model for clean intersections.
\begin{lemma}
\label{lem:lmClean}
Let $M$ be a complex symplectic manifold, of complex dimension $2n$. Let $L_0$ and $L_1$ be complex Lagrangian submanifolds of $M$, and $x \in L_0 \cap L_1$ a point where they intersect cleanly, along a submanifold of complex dimension $k$. Then, there is a neighborhood $S$ of $x$ in $M$ and an isomorphism $h: S \to T^*U$, where $U$ is a neighborhood of $0$ in $\C^n$, such that $h(L_0 \cap S) = U$ and $h(L_1 \cap S)$ is the graph of $df$, where 
$$f: U \to \C, \ \ f(x_1, \dots, x_n) = x_{k+1}^2 + \dots +x_n^2.$$
\end{lemma}

\begin{proof}
Because of the clean intersection condition, we can find a linear isomorphism that takes $TM$ to $\C^{2n}$, the tangent space $TL_0 \subset TM$ to $\C^n \times \{0\}^n \subset \C^{2n}$, and $TL_1 \subset TM$ to the graph of $g: \C^n \to \C^n$, $g(x_1, \dots, x_n)=(0, \dots, 0, x_{k+1},\dots, x_n).$
We then extend this isomorphism to a local neighborhood, as in the proof of Darboux's theorem.
\end{proof}

We now turn to studying Bussi's perverse sheaf $\Pb_{L_0, L_1}$ over a clean intersection.

\begin{proposition}
\label{prop:BClean}
Let $M$ be a complex symplectic manifold, of complex dimension $2n$. Let $L_0$ and $L_1$ be complex Lagrangian submanifolds of $M$, equipped with spin structures. Let $Q \subset L_0 \cap L_1$ be a component of the intersection along which $L_0$ and $L_1$ meet cleanly. Denote by $k$ the complex dimension of $Q$. Then, the restriction of $\Pb_{L_0, L_1}$ to $Q$ is a local system on $Q$ with stalks isomorphic to $\Z[k]$.
\end{proposition}

\begin{proof}
Using Lemma~\ref{lem:lmClean}, we can find an $L_0$-chart $(S, P, U, f, h, i)$ around any $x \in Q$ such that locally the function $f$ is as in Example~\ref{ex:PVclean}. Using the computation of $\PVb_{U,f}$ in that example, and the defining property \eqref{eq:omegaL0} of $\Pb_{L_0, L_1}$, the conclusion follows.
\end{proof}

Our next task is to develop tools for identifying the local system that we obtain from Proposition~\ref{prop:BClean}. 

Under the hypotheses of that proposition, observe that $Q$ is an isotropic submanifold, so we have an isomorphism:
\begin{equation}
  TM|_Q \cong T Q \oplus T^*Q \oplus  N_0 Q \oplus N_1 Q,
\end{equation}
where $N_iQ$ is the normal bundle to $Q$ in $L_i$. In fact, we can identify $T^*Q$ with a complex, but not necessarily holomorphic, isotropic sub-bundle of $TM|_Q$, transverse to $TQ \oplus N_0 Q \oplus N_1 Q.$ (There is a contractible set of choices for such a sub-bundle, just as in the Lagrangian case.)

The direct sum
\begin{equation}
N Q \cong   N_0 Q \oplus N_1 Q,
\end{equation}
is the \emph{symplectic normal bundle} of $Q$, and $N_i Q \ (i=0,1)$ form transverse Lagrangian sub-bundles of $NQ$.

Suppose that the complex bundle $NQ$ has a (not necessarily holomorphic) polarization, transverse to $N_0Q$ and $N_1Q$. This gives a decomposition
\begin{equation} \label{eq:linear_polarization_normal_intersection}
  NQ \cong N_0 Q \oplus N^*_0 Q.
\end{equation}
 The induced projection
\begin{equation} \label{eq:map_normal_bundles}
N_1 Q \to N_0 Q  
\end{equation}
is an isomorphism of complex vector bundles. We obtain a non-degenerate (complex) quadratic form $q$ on $N_0 Q$ such that the graph of $dq$ gives the inclusion
\begin{equation}
  N_1 Q  \subset N_0 Q \oplus N^*_0 Q.   
\end{equation}
By passing to the real part, we obtain a quadratic form on $N_0 Q$ of trivial signature. Let $W^+ \subset N_0 Q$ denote a maximal real sub-bundle on which this form is positive. The space of such sub-bundles is contractible, and therefore the isomorphism class of $W^+$ depends only on the quadratic form. Let $o(W^{+})$ be the $\Z_2$-principal bundle over $Q$ parametrizing orientations of $W^+$. Letting $\Z_2$ act on $\Z$ by $a \mapsto -a$, we define 
$$  |W^{+}| := o(W^+) \otimes_{\Z_2} \Z.$$ 
This is a $\Z$-local system over $Q$.

Observe also that by taking the direct sum of \eqref{eq:map_normal_bundles} with the identity on $TQ$, we obtain an isomorphism
\begin{equation} \label{eq:TLL}
TL_1|_Q \to TL_0|_Q.  
\end{equation}

\begin{lemma} \label{lem:vanish-cycl-desc-1}
If the projection \eqref{eq:TLL} preserves spin structures, we have a canonical isomorphism:
\begin{equation} \label{eq:compute_local_system_perverse}
 \Pb_{L_0,L_1}|_{Q}  \cong |W^{+}| [k].
\end{equation}
\end{lemma}
\begin{proof}
The given polarization of $NQ$ induces a polarization on $TM$, with kernel $T^*Q \oplus N_0^*Q$. Near every point $x \in Q$, we can apply Theorem~\ref{thm:pol1} (b) to obtain from this polarization an $L_0$-chart $(P, U, f, i)$, as explained in Section~\ref{sec:Bussi}. Recall that $ \Pb_{L_0,L_1}|_P$ is naturally isomorphic to
$$ i^*(\PVb_{U,f}) \otimes_{\Z_2} Q_{P, U, f, i}.$$
As in the proof of Proposition~\ref{prop:BClean}, we can choose $f$ to be a quadratic form on $N_0Q$. At $x$, this can be identified with the quadratic form on $T_xL_0 = T_xQ \oplus (N_0Q)_x$ that depends only on the $(N_0Q)_x$ coordinates, where it is given by the form $q$ coming from  \eqref{eq:map_normal_bundles}. Thus, the stalk of $i^*(\PVb_{U,f})$ at $x$ is canonically $H^{k-1}(q^{-1}(\epsilon))$ (shifted to be in degree $-k$) for a small $\epsilon \neq 0$. (Compare Example~\ref{ex:PVclean}.) The preimage $q^{-1}(\epsilon)$ is (non-canonically) diffeomorphic to $T^*S^{k-1}$, and an identification of $H^{k-1}(q^{-1}(\epsilon))$ with $\Z$ is the same as a choice of an orientation on $W^+$ at $x$, or of an identification of $|W^+|_x$ to $\Z$. Thus, we have a canonical isomorphism
$$i^*(\PVb_{U,f})_x \cong  |W^{+}|_x [k].$$
Moreover, because of the condition on spin structures, the bundle $Q_{P, U, f, i}$ has a canonical section, so tensoring it with it has no effect. The conclusion follows. 
\end{proof}

As a consequence of Lemma \ref{lem:vanish-cycl-desc-1}, to compute $\Pb_{L_0,L_1}$ one needs to find a polarization in the symplectic normal bundle of $Q$, in which the two spin structures and the quadratic form can be explicitly understood. An example of such a situation will appear in Lemma~\ref{lem:trivial} below.

\section{Three-manifold invariants}
\label{sec:Invariant}

In this section we construct the three-manifold invariants advertised in the Introduction, and prove Theorems~\ref{thm:main1} and \ref{thm:framed}.

\subsection{Definitions}
Let $Y$ be a closed, connected, oriented three-manifold. Suppose we are given a Heegaard splitting $Y = U_0 \cup_{\Sigma} U_1$ of genus $g \geq 3$. We equip $\CharIrr(\Sigma)$ with the complex structure $J$ and the complex symplectic form $\omegac=-\omega_1 + i \omega_3$, as in Section~\ref{sec:surfaces}. Let $$L_0, L_1 \subset \CharIrr(\Sigma)$$ be the complex Lagrangians constructed in Section~\ref{sec:Lags}. By Lemma~\ref{lem:idab}, the intersection $X=L_0 \cap L_1$ can be identified with $\CharIrr(Y)$. Further, each $L_i$ is diffeomorphic to $\CharIrr(F_g)$, where $F_g$ is the free group on $g$ elements. By Lemma~\ref{lem:pifree}, we have $H^2(L_i; \Z/2)=0$, so $c_1(TL_i)=0$ and hence $L_i$ admits a spin structure, which is unique because $H^1(L_i; \Z/2) = 0$; cf. Fact~\ref{fact:spin}. Applying Bussi's work described in Section~\ref{sec:Bussi}, we obtain a perverse sheaf
$$P^{\bullet}(Y):=P^{\bullet}_{L_0, L_1} \in \Perv(\CharIrr(Y)).$$

We can do a framed version of this construction, using the complex Lagrangians $$L_0^{\#} , L_1^{\#} \subset  \Char_{\tw}(\Sigma^{\#})$$
constructed in Section~\ref{sec:LagTw}. In this case we can use a Heegaard splitting of any genus $g \geq 0$, and we need to pick a basepoint $z \in \Sigma \subset Y$. The Lagrangians $L_0^{\#} , L_1^{\#} $ are diffeomorphic to products of $g$ copies of $G \cong S^3 \times \R^3$, so they too have unique spin structures. We let
$$ P^{\bullet}_\#(Y,z):=P^{\bullet}_{L_0^\#, L_1^\#} \in \Perv(\Rep(Y)).$$

When relating the perverse sheaves coming from different Heegaard splittings, we may encounter nontrivial self-diffeomorphisms of $Y$, which in turn give nontrivial automorphisms (self-biholomorphisms) of $\CharIrr(Y)$ and $\Rep(Y)$. Therefore, it will be helpful to work in the following variant of the category of perverse sheaves.

\begin{definition}
\label{def:pervp}
If $X$ is a complex analytic space, we let $\Perv'(X)$ be the category whose objects are the same as in $\Perv(X)$, and whose morphisms are defined as follows. If $\Cat^\bullet$ and $\D^\bullet$ are perverse sheaves on $X$, a morphism from $\Cat^\bullet$ and $\D^\bullet$ in $\Perv'(X)$ is a pair $(f, \phi)$, where $f: X \to X$ is an automorphism and $\phi: \Cat^{\bullet} \to f^*\D^\bullet$ is a morphism in $\Perv(X)$. Composition of morphisms is given by
$$ (f, \phi) \circ (g, \psi) := (f \circ g, g^*\phi \circ \psi).$$
\end{definition}

We will prove below that $P^{\bullet}(Y)$ is a natural invariant of $Y$ in the category $\Perv'(\CharIrr(Y))$, and $ P^{\bullet}_\#(Y,z)$  is a natural invariant of the pair $(Y,z)$ in the category $\Perv'(\Rep(Y))$. By taking hypercohomology, we will then obtain invariants
$$ \HP^*(Y),  \ \ \HPf^*(Y,z),$$
as noted in the Introduction.

\subsection{Stabilization invariance}
\label{sec:stab}
When describing Heegaard splittings of three-manifolds, it will be convenient to use Heegaard diagrams, as in Heegaard Floer theory \cite{HolDisk}. Specifically, we represent the handlebody $U_0$ with boundary $\Sigma$ by a collection of $g$ disjoint simple closed curves $\alpha_1, \dots, \alpha_g$ on $\Sigma$, homologically independent in $H_1(\Sigma)$, such that $U_0$ is obtained from $\Sigma$ by attaching disks with boundaries $\alpha_i$, and then attaching a three-ball. Similarly, we represent $U_1$ by another collection of curves, denoted $\beta_1, \dots, \beta_g$. The data
$$ (\Sigma, \alpha_1, \dots, \alpha_g, \beta_1, \dots, \beta_g)$$
is a {\em Heegaard diagram}. 

Note that our constructions of $P^{\bullet}(Y)$ and $P^{\bullet}_\#(Y,z)$ start directly from a Heegaard splitting, not a Heegaard diagram. Thus, unlike in Heegaard Floer theory, to prove invariance there will be no need to consider moves that change the Heegaard diagram but leave the splitting fixed. (These Heegaard moves are the handleslides and curve isotopies, considered in \cite{HolDisk}.) For us, Heegaard diagrams will be just a way of representing Heegaard splittings pictorially, as in Figures~\ref{fig:stabilization}, \ref{fig:handleswap}, and \ref{fig:handleswap2} below.

The one Heegaard move that we have to consider is {\em stabilization}. This consists in drilling out a solid torus from one of the handlebodies, say $U_1$, such that a part of its boundary (a disk $D$) is on $\Sigma$, and then attaching the solid torus to $U_0$. In this way, from the Heegaard splitting $(\Sigma, U_0, U_1)$ of genus $g$ we obtain a new Heegaard splitting $(\Sigma', U_0', U_1')$ of genus $g+1$, for the same three-manifold $Y$. In terms of Heegaard diagrams, we have introduced two new curves $\alpha'$ and $\beta'$, intersecting transversely in one point. See Figure~\ref{fig:stabilization}.

\begin {figure}
\begin {center}
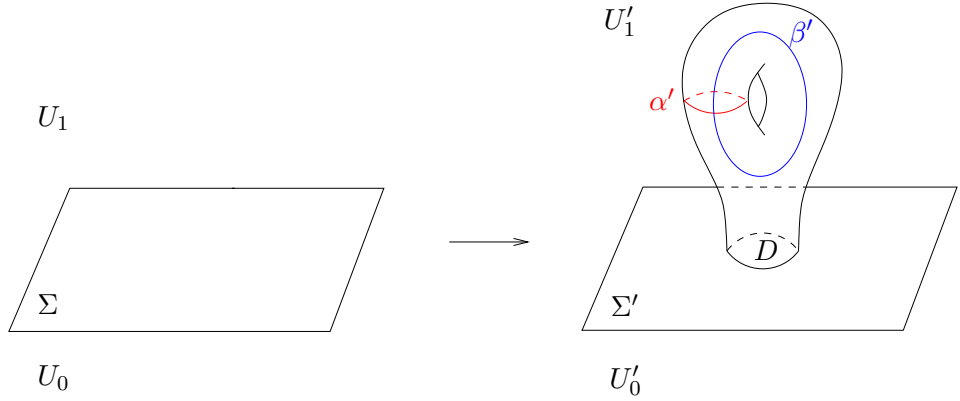
\caption {Stabilization.}
\label{fig:stabilization}
\end {center}
\end {figure}

The inverse move to a stabilization is called {\em destabilization}.

\begin{theorem}[Reidemeister \cite{Reidemeister}, Singer \cite{Singer}]
\label{thm:RS}
Given a three-manifold $Y$, any two Heegaard splittings for $Y$ are  related by a sequence of stabilizations and destabilizations.
\end{theorem}

\begin{remark}
\label{rem:ambient}
We view all our Heegaard surfaces not just as abstract surfaces, but as submanifolds of $Y$.  Changing the Heegaard surface by an ambient isotopy in $Y$ could be considered another Heegaard move (relating different Heegaard splittings). However, note that we can obtain a small ambient isotopy by composing a stabilization (perfomed from a disk $D \subset \Sigma$) with a destabilization that collapses the solid torus to a new disk $D'$, a slight deformation of $D$. Therefore, every ambient isotopy is a composition of  stabilizations and destabilizations.  

Even more generally, we could consider diffeomorphisms $f: Y \to Y$ that take a Heegaard splitting to another one, and are not necessarily isotopic to identity. Once again, these are not necessary if we  want to relate different Heegaard splittings of $Y$.
\end{remark}

In view of Theorem~\ref{thm:RS}, in order to prove that the isomorphism classes of $P^{\bullet}(Y)$ and $P^{\bullet}_\#(Y,z)$ are invariants of $Y$, resp. $(Y, z)$, it suffices to consider the effect of stabilizations.

\begin{proposition}
\label{prop:stab}
Let $(\Sigma, U_0, U_1)$ be a Heegaard splitting for $Y$, and $(\Sigma', U_0', U_1')$ be obtained from it by a stabilization. Let $L_0, L_1  \subset \CharIrr(\Sigma)$ and $L_0', L_1' \subset \CharIrr(\Sigma')$ be the complex Lagrangians constructed from each Heegaard splitting as in Section~\ref{sec:Lags}, and $L_0^{\#} , L_1^{\#} \subset  \Char_{\tw}(\Sigma^{\#})$, ${L'_0}^{\#} , {L'_1}^{\#} \subset  \Char_{\tw}({\Sigma'}^{\#})$ be those constructed as in Section~\ref{sec:LagTw}.

Then, the stabilization move induces isomorphisms
$$ \Stab: \Pb_{L_0, L_1} \xrightarrow{\phantom{b}\cong\phantom{b}}  \Pb_{L_0', L_1'} \ \text{in } \Perv(\CharIrr(Y))$$
and
$$ \Stab^{\#}:  \Pb_{L^{\#}_0, L^{\#}_1}  \xrightarrow{\phantom{b}\cong\phantom{b}}  \Pb_{{L_0'}^{\#}, {L_1'}^{\#}} \ \text{in } \Perv(\Rep(Y)).$$
\end{proposition}

\begin{proof}
 To construct $\Stab$, we apply Proposition~\ref{prop:stabL}. We take $M=\CharIrr(\Sigma)$ and $M' = \CharIrr(\Sigma')$. There is a projection
$$ \pi_1(\Sigma') \cong \pi_1(\Sigma \setminus D) *_{\pi_1(\del D)} \pi_1(T^2 \setminus D) \longrightarrow \pi_1(\Sigma)$$
given by sending the generators of $\pi_1(T^2 \setminus D)$ to $1$. This induces an inclusion $\Phi: M \hookrightarrow M'$. If we describe $M$ in terms of the holonomies of flat connections, as in \eqref{eq:CS}, and do the same for $M'$, with the holonomies around $\alpha'$ and $\beta'$ being $A'$ and $B'$, then $M \subset M'$ is given by the equations
$$ A' = B' = I.$$
(Note that to define the holonomies $A'$ and $B'$, we need to choose a basepoint on the respective curves and an identification of the fiber at that point with $\C^2$. However, the condition that a holonomy is trivial is invariant under conjugation, and hence independent of those choices.)

Let us push $\Sigma$ slightly inside $U_0$ and consider the compression body $Z_0$ situated between this new copy of $\Sigma$ and $\Sigma'$.  Then $\pi_1(Z_0) \cong \pi_1(\Sigma) * \langle \beta' \rangle.$ Let us denote by $C_0 \subset \CharIrr(Z_0)$ the space of representations of $\pi_1(Z_0)$ whose restriction to $\pi_1(\Sigma)$ is irreducible. An argument similar to that in the proof of Lemma~\ref{lem:idab}(b) shows that $C_0$ is a coisotropic complex submanifold of $M'=\CharIrr(\Sigma')$. We have inclusions
$$ M \subset C_0 \subset M'$$
with $C_0$ being given by the equation $A'=I$. Observe also that $L_0 = L_0' \cap C_0 \subset M,$ and that we have an isomorphism
$$ C_0 \cong \RepIrr(\Sigma) \times_G G,$$
where $G$ acts on itself by conjugation. From here we see that  $C_0$ is a $G$-bundle over $M= \RepIrr(\Sigma)/G$. (This is not a principal bundle.) The $G$-bundle comes with a canonical section 
$$ M \to C_0, \ \ [\rho] \mapsto [(\rho, 1)],$$
which gives the inclusion $M \subset C_0$ mentioned above. Note that the tangent bundle to $C_0$ at a point $[(\rho, 1)] \in M$ is $T(\RepIrr(\Sigma) \times \g)/\g,$ where the denominator $\g$ is the tangent bundle to the orbit of $(\rho, 1)$. This orbit lies in $\RepIrr(\Sigma) \times \{1\}$, and therefore 
we can identify $TC_0|_M$ with $TM \times \g$.

Thus, if we let $V_0$ be the symplectic complement to $TM$ inside $TC_0|_{M}$, then $V_0$ is isomorphic to the trivial $\g$-bundle over $M$.  The Killing form on $\g$ gives a non-degenerate holomorphic quadratic form $q \in H^0(\Sym^2(V_0^*))$.

Let us also consider a compression body $Z_1$ between $\Sigma$ and $\Sigma'$, obtained by compressing $\beta'$ instead of $\alpha'$. This gives rise to another coisotropic $C_1 \subset M'$, determined by the equation $B'=I$. We have $M \subset C_1$ and $L_1 = L_1' \cap C_1$. We let $V_1$ be the symplectic complement to $TM$ inside $TC_1$.  Clearly, we have
$$ TL_0'|_{L_0} = TL_0 \oplus V_0|_{L_0} \oplus 0 \subset TM|_{L_0} \oplus V_0|_{L_0} \oplus V_1|_{L_0} =TM'|_{L_0}.$$

In fact, we can naturally identify the normal bundle $N_{MM'}$ with the trivial bundle with fiber $H^1(T^2; \g)$, where $T^2$ is the torus introduced in the stabilization. Then, $V_0 \subset N_{MM'}$ is spanned by the Poincar\'e dual to $\alpha'$, and $V_1$ by the Poincar\'e dual to $\beta'$.

The bundles $V_0$ and $V_1$ are trivial, so they admit spin structures. Further, these spin structures are unique, by Fact~\ref{fact:spin}(b), because the base space $M$ is simply connected, and therefore $H^1(M; \Z/2)=0$. To see that $M=\CharIrr(\Sigma)$ is simply connected, one can imitate the Morse-theoretic proof given by Hitchin in \cite[Theorem 9.20]{Hitchin} for the space $\Xtw(\Sigma)$; compare \cite[Section 4]{DWW}. 

Since the spin structures on $V_0$ and $V_1$ are unique, they correspond to each other under the duality induced by $\omega$. Furthermore, the spin structure on $V_0$ is self-dual under the isomorphism induced by $q$. Recall also that the Lagrangians $L_i$ and $L_i'$ have unique spin structures, so these must be compatible with the ones on $V_0$ and $V_1$.

We conclude that the hypotheses of Proposition~\ref{prop:stabL} are satisfied. We let $\Stab$ be the resulting isomorphism.

The isomorphism $\Stab^{\#}$ is constructed in a similar manner. The role of $Z_0$ is played by a compression body $Z_0^{\#}$ between $\Sigma^{\#}$ and ${\Sigma'}^{\#}$, and we use the coisotropic submanifold $C^{\#}_0 = \Xtw(Z^{\#}_0) \subset \Xtw(\Sigma^{\#})$.
\end{proof}

\subsection{Naturality}
\label{sec:naturality}
Proposition~\ref{prop:stab}, combined with Theorem~\ref{thm:RS}, shows that $P^{\bullet}(Y)$ and $P^{\bullet}_\#(Y)$ are invariants of $Y$ up to isomorphism. To complete the proofs of Theorems~\ref{thm:main1} and \ref{thm:framed}, we still have to show that they are {\em natural} invariants, i.e., that the isomorphisms can be chosen canonically. Specifically, given two Heegaard splittings of $Y$, we can relate them by a sequence of moves, and thus get an isomorphism between the objects constructed from each Heegaard splitting. The naturality claim is that this isomorphism does not depend on the chosen sequence of moves. (For the framed invariant $P^{\bullet}_\#(Y)$, we expect dependence on the basepoint $z$, so we will only consider moves that keep $z$ fixed.)

Naturality for three-manifold invariants defined from Heegaard diagrams was studied by Juh\'asz, D.Thurston and Zemke in \cite{JuhaszThurston}, where they applied it to Heegaard Floer homology. Theorem 2.39 in \cite{JuhaszThurston} gives a finite list of conditions that need to be checked to ensure naturality; see \cite[Definitions 2.30 and 2.33]{JuhaszThurston}. In our context, the invariants are constructed directly from a Heegaard splitting, so the list is shorter. Indeed, we can view invariants defined from a Heegaard splitting as being defined from a Heegaard diagram, with the $\alpha$-equivalence and $\beta$-equivalence moves from \cite{JuhaszThurston} inducing the identity. Thus, for our purposes, we will only consider the following Heegaard moves: stabilizations, destabilizations, and diffeomorphisms. Diffeomorphisms are not strictly necessary, cf. Remark~\ref{rem:ambient}. However, we will include them to keep the statements cleaner and more in line with  \cite{JuhaszThurston}. We will write a diffeomorphism $f: Y \to Y$ that takes a Heegaard splitting $\cH$ to another one $\cH'$ as $f: \cH \to \cH'$. A particular role will be played by diffeomorphisms that are isotopic to the identity in $Y$.

Before stating the naturality result, let us recall the notion of {\em simple handleswap}, which plays an essential role in \cite{JuhaszThurston}. Let $\cH=(\Sigma, U_0, U_1)$ be a Heegaard splitting. Let $D', D'' \subset \Sigma$ be the disks bounded by the curves $c'$ and $c''$ shown in Figure~\ref{fig:handleswap}. By adding the handle $H'$ to $U_1$, we get a new Heegaard splitting $\cH'= (\Sigma', U_0', U_1')$. We view this operation as the composition
$$ e = e_{\st} \circ e_{\iso} : \cH \to \cH',$$
where $e_{\iso}$ is a small isotopy given by pushing $D'$ slightly into $U_1$, to get a new disk bounded by $c'$, and $e_{\st}$ is the stabilization given by attaching a solid torus (the union of the handle $H'$ with the region $R$ between $D'$ and the new disk) to $U_1 \setminus R$. In a similar manner, we add a handle $H''$ to $U_0'$ to get the splitting $\cH'' =(\Sigma'', U_0'', U_1'')$.  The operation
$$ f = f_{\st} \circ f_{\iso} : \cH' \to \cH''$$
is the composition of an isotopy $f_{\iso}$ (pushing $D''$ into $U_0$) and a stabilization $f_{\st}$.

\begin {figure}
\begin {center}
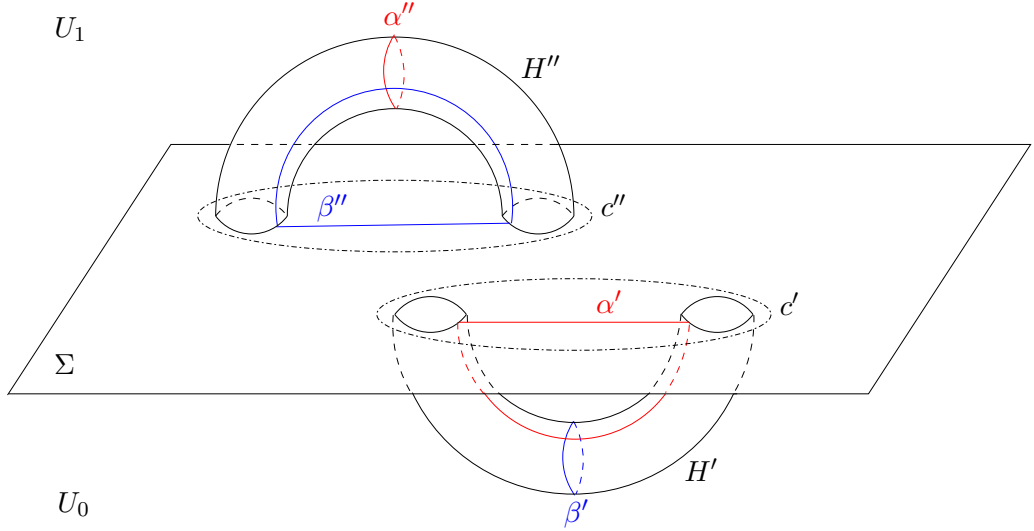
\caption {We draw a part of the surface $\Sigma$ of genus $h$ as the middle plane (without the handles), and $U_0$ and $U_1$ as the lower and upper half-space, respectively. We drill a handle $H'$ into $U_0$ to obtain a Heegaard decomposition $(\Sigma', U_0', U_1')$, of genus $h+1$. Then we add a handle $H''$ to $U_0'$ as shown, and we obtain a new Heegaard decomposition $(\Sigma'', U_0'', U_1'')$, of genus $h+2$. }
\label{fig:handleswap}
\end {center}
\end {figure}

Now, on the surface $\Sigma$, we consider the diffeomorphism 
$$g = \tau_{\gamma} \circ \tau_{\gamma'}^{-1} \circ \tau_{\gamma''}^{-1} : \Sigma \to \Sigma$$  
given by the composition of a right-handed Dehn twist along the curve $\gamma$ and left-handed Dehn twists along the curves $\gamma'$ and $\gamma''$ shown in Figure~\ref{fig:handleswap1}. This maps the curves $\alpha'$ to $\hat{\alpha}'$ and $\beta''$ to $\hat{\beta}''$.

\begin {figure}
\begin {center}
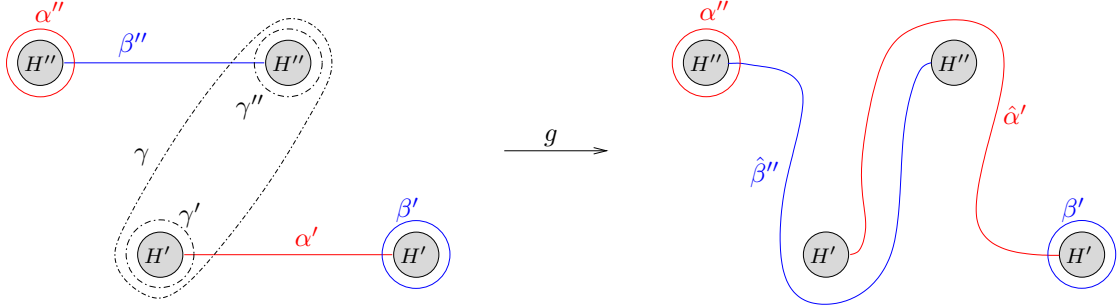
\caption {On the left we drew the part of the surface $\Sigma$ from Figure~\ref{fig:handleswap}, with the gray circles being the feet of the respective handles. On the right we drew the effect of the diffeomorphism $g$ on the given curves.}
\label{fig:handleswap1}
\end {center}
\end {figure}

\begin{remark}
Figure~\ref{fig:handleswap1} should be compared to Figure 4 in \cite{JuhaszThurston}. Our curves $\alpha', \beta', \alpha'', \beta''$ play the roles of $\alpha_2$, $\beta_1$, $\alpha_1$ and $\beta_2$ in their notation. Their set-up also involves an $\alpha$-equivalence and a $\beta$-equivalence, but in our case these act by the identity. Our diffeomorphism $g$ is the inverse of the one considered there.
\end{remark}

We extend $g$ to a diffeomorphism $g: Y \to Y$ as follows. Consider the disk enclosed by the curve $\gamma$ in Figure~\ref{fig:handleswap1}, and enlarge it slightly to obtain a disk $D$ that contains $\gamma$ in its interior. Let $T' = H' \cap D$ and $T'' = H'' \cap D$ be the feet of the handles contained in $D$. Let also $U= D \times [-1,1]$ be a three-dimensional cylindrical neighborhood of $D$ in $Y$, which intersects $\Sigma$ at $D= D \times \{0\}$, as in Figure~\ref{fig:cylinder}, with 
$$H' \cap U = T' \times [-1,0], \ \ H'' \cap U =  T'' \times [0,1].$$

\begin {figure}
\begin {center}
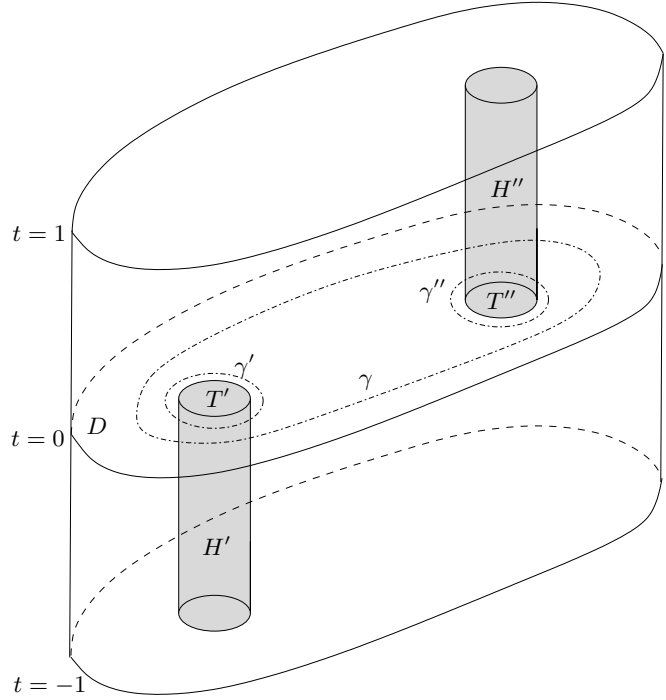
\caption {A three-dimensional neighborhood $U$ of the disk enclosed by $\gamma$. The parts of the handles $H'$ and $H''$ contained in $U$ are in grey.}
\label{fig:cylinder}
\end {center}
\end {figure}

Observe that the diffeomorphism $g$, when restricted to $D \setminus (T' \cup T'')$, is not isotopic to the identity rel boundary. However, if we restrict it to $D \setminus T'$, it is isotopic to the identity rel boundary. This is because the Dehn twist around $\gamma''$ is isotopic to the identity when we can go over $T''$, and  the Dehn twists along $\gamma'$ and $\gamma$ are in opposite directions, so they cancel each other  out. By following the isotopy from $g|_{D \setminus T'}$ to the identity in each slice $D \times \{t\}, t \in [-1,0]$, we extend $g$ to a diffeomorphism from $D \times [-1,0]$, which acts by the identity on $D \times \{-1\}$ and on the grey cylinder $T' \times [0,1]$. Similarly, we extend $g$ to the upper half $D \times [0,1] \subset U$, using an isotopy from $g|_{D \setminus T''}$ to the identity. We obtain a diffeomorphism $g: U \to U$, which is the identity on $\del U = (\del D \times [-1,1]) \cup (D \times \{-1,1\})$ and on the two grey cylinders. We then extend $g$ to a diffeomorphism $g:Y \to Y$, by the identity outside $U$. 

Note that $g: U \to U$ preserves the Heegaard splittings $\cH, \cH'$ and $\cH''$. Observe also that the restrictions of $g$ to $\Sigma$ and $\Sigma'$ are isotopic to the identity. However, this is not the case for $\Sigma''$. We refer to 
$$ g: \cH'' \to \cH''$$
as a {\em simple handleswap}. 

The following definition is a variant of \cite[Definition 2.32]{JuhaszThurston}, adapted to our setting where the constructions are done starting directly from Heegaard splittings. Also, for simplicity, we restrict ourselves to invariants associated to a given manifold $Y$, rather than to a class of diffeomorphism types as in  \cite{JuhaszThurston}.

\begin{definition}
\label{def:JT}
Let $Y$ be a closed, connected, oriented three-manifold, and $\Cat$ a category. A {\em strong Heegaard invariant} $F$ of $Y$ consists of:
\begin{itemize}
\item an assignment to every Heegaard splitting $\cH$ of $Y$ of an object $F(\cH) \in \Cat$, and
\item to every Heegaard move $e$ (stabilization, destabilization, or diffeomorphism) between two splittings $\cH_1$ and $\cH_2$, an assignment of a morphism $F(e): F(\cH_1) \to F(\cH_2)$.
\end{itemize}

Furthermore, these morphisms are required to satisfy the following properties:
\begin{enumerate}
\item {\bf Functoriality:} 
\begin{enumerate}[(i)]
\item If $e: \cH_1 \to \cH_2$ and $f: \cH_2 \to \cH_3$ are diffeomorphisms, then for the combined diffeomorphism $f \circ e: \cH_1 \to \cH_3$, we have $F(f \circ e) = F(f) \circ F(e)$. 
\item If $e : \cH_1 \to \cH_2$ is a stabilization and $e': \cH_2 \to \cH_1$ is the corresponding destabilization, then $F(e') = F(e)^{-1}$.
\end{enumerate}

\item {\bf Commutativity:}
\begin{enumerate}[(i)]
\item If $e : \cH_1 \to \cH_2$ and $g: \cH_2 \to \cH_4$ are stabilizations given by adjoining disjoint solid tori $H_1$ resp. $H_2$, and $f: \cH_1 \to {\cH}_3$, $h:{\cH}_3 \to \cH_4$ are stabilizations given by attaching $H_2$ resp. $H_1$, then $F(h) \circ F(f) = F(g) \circ F(e)$.
\item If $e : \cH_1 \to \cH_2$ is a stabilization and  $f: \cH_1 \to \cH_3$ is a diffeomorphism, let $g: \cH_2 \to \cH_4$ be the same diffeomorphism as $f$ but acting on the stabilized surface, and $h: \cH_3 \to \cH_4$ the corresponding stabilization (the image of $e$ under $f$). Then, $F(h) \circ F(f) = F(g) \circ F(e)$.
\end{enumerate}

\item {\bf Continuity:} If $e: \cH \to \cH$ is a diffeomorphism such that $e|_{\Sigma}: \Sigma \to \Sigma$ is isotopic to $\id_{\Sigma}$, then $F(e) = \id_{F(\cH)}$.

\item {\bf Handleswap invariance:} Given a simple handleswap $g: \cH'' \to \cH''$ as in Figure~\ref{fig:handleswap1}, we ask that $F(g) = \id_{F(\cH'')}$.
\end{enumerate}
\end{definition}

If $Y$ is as in Definition~\ref{def:JT} and $z \in Y$ is a basepoint, we can define a strong Heegaard invariant of the pair $(Y, z)$ in a similar way, by considering only Heegaard splittings with $z$ on the Heegaard surface, and Heegaard moves that fix $z$.

The following is a rephrasing of Theorem 2.38 in \cite{JuhaszThurston} in our context. 

\begin{theorem}[Juh\'asz-Thurston-Zemke \cite{JuhaszThurston}]
\label{thm:JT}
Let $F$ be a strong Heegaard invariant of a three-manifold $Y$, with values in a category $\Cat$. 
Then, for any two Heegaard splittings $\cH, \cH'$, if we relate them by a sequence of Heegaard moves involving only stabilizations, destabilizations, and diffeomorphisms isotopic to the identity in $Y$
$$\cH =\cH_0 \xrightarrow{\phantom{b} e_1 \phantom{b}} \cH_1 \xrightarrow{\phantom{b} e_2 \phantom{b}} \dots  \xrightarrow{\phantom{b} e_n \phantom{b}} \cH_n = \cH',$$
the induced morphism
$$ F(\cH, \cH')= F(e_n) \circ \dots \circ F(e_1) : F(\cH) \to F(\cH')$$
depends only on $\cH$ and $\cH'$, and not on the sequence of moves chosen to relate them.

Moreover, the same naturality result works for based three-manifolds $(Y, z)$, if we consider only Heegaard splittings with $z$ on the Heegaard surface, and Heegaard moves that fix $z$.
\end{theorem}

Note that the output of Theorem~\ref{thm:JT} is the set of isomorphisms $F(\cH, \cH') : F(\cH) \to F(\cH')$ satisfying
\begin{itemize}
\item $F(\cH, \cH) = \id_{F(\cH)}$ for every $\cH$;
\item $F(\cH', \cH'') \circ F(\cH', \cH) = F(\cH, \cH'')$ for every $\cH, \cH'$ and $\cH''$.
\end{itemize}

\begin{remark}
If $\Cat$ is the category of groups (or Abelian groups), then the data consisting of the groups $F(\cH)$ and the isomorphisms $F(\cH, \cH')$ (satisfying the two properties above) is called a {\em transitive system of groups}, in the terminology of Eilenberg-Steenrod \cite[Definition 6.1]{EilenbergSteenrod}.
Given a transitive system of groups such as this, we obtain a single group $G$ with elements $g \in \prod_{\cH} F(\cH)$ such that $F(\cH, \cH')(g(\cH)) = g(\cH')$, for all $\cH, \cH'$. Thus, under the hypotheses of Theorem~\ref{thm:JT}, we obtain a group $G$ that is associated to the three-manifold $Y$. 

In our setting, once we establish naturality, we can apply this construction to define the hypercohomology invariants $ \HP^*(Y)$ and $\HPf^*(Y)$ as graded Abelian groups associated to $Y$ (independent of any choices, except for the basepoint $z$ for the framed versions).
\end{remark}

We now seek to apply Theorem~\ref{thm:JT} to the objects $$F(\cH)= \Pb_{L_0, L_1} \in \Perv'(\CharIrr(Y)),$$ defined from Heegaard splittings, where $\Perv'(\CharIrr(Y))$ is the category introduced in Definition~\ref{def:pervp}. For this, we first need to specify the maps $F(e)$. When $e$ is a stabilization, we use the isomorphism $\Stab$ constructed in Proposition~\ref{prop:stab}; for the corresponding destabilization, we use the inverse of $\Stab$. When $e$ is a diffeomorphism taking the Heegaard splitting $\cH=(\Sigma, U_0, U_1)$ to $\cH'=(\Sigma', U_0', U_1')$, observe that $e$ induces an isomorphism between the complex symplectic manifolds $\CharIrr(\Sigma)$ and $\CharIrr(\Sigma')$, taking the corresponding Lagrangians $L_0, L_1$ into $L_0', L_1'$. From here we obtain an automorphism of $\CharIrr(Y)$ and an isomorphism $F(e): \Pb_{L_0, L_1} \to \Pb_{L_0', L_1'}$ in the category $\Perv'(\CharIrr(Y))$.

The proof of Theorem~\ref{thm:main1} will be complete once we establish the following.

\begin{proposition}
The objects $F(\cH)= \Pb_{L_0, L_1} \in \Perv'(\CharIrr(Y))$ and the maps $F(e)$ defined above satisfy the hypotheses of Theorem~\ref{thm:JT}. Hence, $\Pb(Y) = \Pb_{L_0, L_1}$ is a natural invariant of $Y$.
\end{proposition}

\begin{proof}
Functoriality and commutativity are immediate from the construction. For continuity, note that the induced action of $\Diff(\Sigma)$ on $\CharIrr(\Sigma)$ factors through the mapping class group $\pi_0(\Diff(\Sigma))$ of $\Sigma$; this is clear when we view the elements of $\CharIrr(\Sigma)$ as conjugacy classes of maps $\pi_1(\Sigma) \to G$. Thus, when $e$ is isotopic to the identity, it must act by the identity on $\CharIrr(\Sigma)$, and hence on the perverse sheaves.

To prove handleswap invariance, let us first reformulate it in terms of stabilizations. With the notation from the definition of a simple handleswap in Theorem~\ref{thm:JT}, we have moves $e = e_{\st} \circ e_{\iso}: \cH \to \cH'$ and $f = f_{\st} \circ f_{\iso}: \cH' \to \cH''$. Let us also consider another similar move $\hat{f}= \hat{f}_{\st} \circ \hat{f}_{\iso}: \cH' \to \cH''$, given by attaching a solid torus to $U_0$ along the disk bounded by the curve $\hat{c}''$ from Figure~\ref{fig:handleswap2}; the effect of this is still adding the handle $H''$, but we choose a different path between its feet to view it as a small  isotopy (push off into $U_0$) plus a stabilization. 

\begin {figure}
\begin {center}
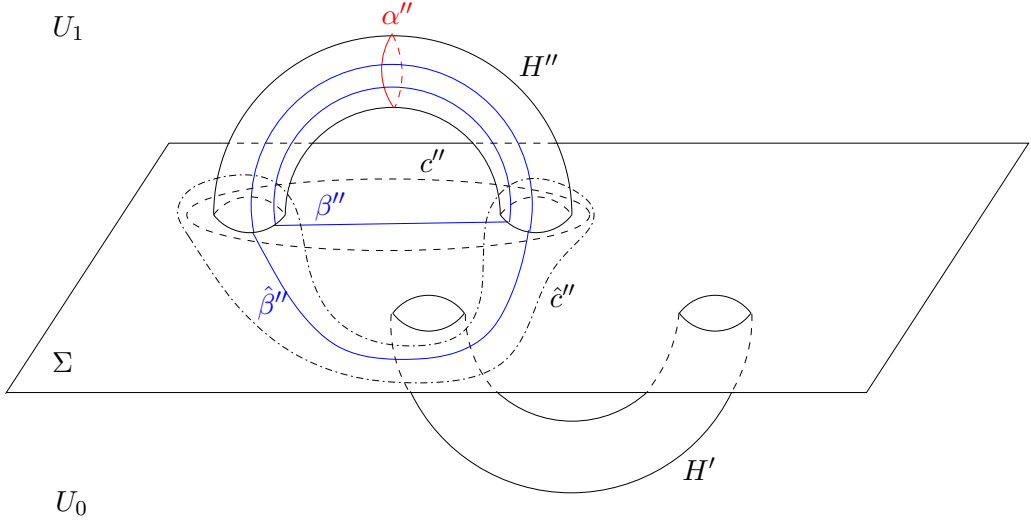
\caption {Adding the handle $H''$ can be viewed as a stabilization in two different ways.}
\label{fig:handleswap2}
\end {center}
\end {figure}

By the commutativity between stabilizations and diffeomorphisms, together with functoriality and continuity for isotopies, we have
$$ F(\hat{f}) \circ F(g') = F(g) \circ F(f),$$
where $g'$ is the same as $g$, but acting on $\cH'$. Recall from the discussion of handleswaps that the restriction of the diffeomorphism $g$ to $\Sigma'$ is isotopic to the identity. By continuity, we must have $F(g') = \id_{F(\cH')}$. Therefore,
$$F(\hat{f}) = F(g) \circ F(f).$$
Thus, the handleswap invariance condition $F(g) = \id_{F(\cH'')}$ is equivalent to 
\begin{equation}
\label{eq:Fff}
 F(\hat{f}) = F(f).
 \end{equation}
In other words, we want that the move from $\cH'$ to $\cH''$ depends only on the handle $H''$, and not on the path joining the feet of $H''$. 

Let us also bring the move $e = e_{\st} \circ e_{\iso}: \cH \to \cH'$ into play. Since $F(e)$ is an isomorphism, the condition \eqref{eq:Fff} is equivalent to
\begin{equation}
\label{eq:Fe}
 F(\hat{f}) \circ F(e)= F(f) \circ F(e) : F(\cH) \to F(\cH'').
 \end{equation}
 
In our context, let us denote
$$M=\CharIrr(\Sigma), \ M' = \CharIrr(\Sigma'), \ M''=\CharIrr(\Sigma'').$$
We let $A', B', A'', B'', \hat{B}''$ denote the holonomies of flat connections around $\alpha', \beta', \alpha'', \beta'', \hat{\beta}''$. With a suitable choice of basepoint, we can arrange so that $\hat{\beta}'' = \beta''\cdot \beta'$ in $\pi_1(\Sigma)$, and therefore $\hat{B}''=B''B'$. 

Then, if we use the curve $c''$ to do the second stabilization, we find that $M'$ sits inside $M''$ as the subset given by $A''=1, B'' =1.$ However, if we use $\hat{c}''$ to do the stabilization, we get another copy of $M'$, which we will call $\widehat{M}'$, given by the subset of $M''$ with $A''=1, B''B' = 1.$ (The two embeddings of $M'$ into $M$ correspond to different projections $\pi_1(\Sigma'') \to \pi_1(\Sigma')$.) Finally, $M \subset M' \cap \widehat{M}' \subset M''$ is given by $A' = B' = A'' = B'' = 1.$

In summary, we have a commutative diagram of embeddings of complex symplectic manifolds:
$$
\xymatrix{
& M'  \ar@{^{(}->}^{\Phi'} [rd] & \\
M \ar@{^{(}->}^{\Phi} [ru]  \ar@{^{(}->}^{\hat{\Phi}} [rd] & & M'' \\
& \widehat{M}'  \ar@{^{(}->}^{\hat{\Phi}'} [ru] &
}$$

The Heegaard splittings give rise to complex Lagrangians
$$ L_0, L_1 \subset M, \ \ L_0', L_1' \subset M', $$
$$ \hat{L}'_0, \hat{L}'_1 \subset \widehat{M}', \ \ L_0'', L_1'' \subset M'',$$
all equipped with (unique) spin structures. We also have coisotropics induced by the compression bodies (as in the proof of Proposition~\ref{prop:stab}), which give decompositions of the normal bundles to each submanifold into holomorphic Lagrangian bundles
\begin{equation}
\label{eq:normals}
 N_{MM'} = V_0 \oplus V_1, \ \ N_{M'M''} = V'_0 \oplus V'_1, \ \ N_{M\widehat{M}'} = \widehat{V}_0 \oplus \widehat{V}_1, \ \ N_{\widehat{M}'M''} = \widehat{V}'_0 \oplus \widehat{V}'_1.
 \end{equation}

Specifically, the normal bundle $N_{MM'}$ to $M$ in $M'$ can be identified with the trivial bundle with fiber $H^1(T^2; \g)$, where $T^2$ is the torus attached in the stabilization from $\Sigma$ to $\Sigma'$. In the decomposition $N_{MM'} = V_0 \oplus V_1$, the first summand $V_0$ is spanned (over $\g$) by the class $a'$ Poincar\'e dual to $[\alpha']$, and the second summand $V_1$ by the class $b'$ Poincar\'e dual to $[\beta']$.

Similarly, the bundle $N_{M'M''}$ decomposes as $V_0' \oplus V_1'$, where $V_0'$ is spanned by  the class $a''$ Poincar\'e dual to $[\alpha'']$, and $V_1'$ by the class $b''$ Poincar\'e dual to $[\beta'']$. The  bundle $N_{M'M''}$ decomposes as $\widehat{V}_0' \oplus \widehat{V}_1'$, where $\widehat{V}_0'$ is spanned by $a''$  and $\widehat{V}_1'$ by $\hat{b}''= b' + b''$ (the image of $b''$ under the handleswap diffeomorphism $g$; cf. Figure~\ref{fig:handleswap1}). Finally, $N_{M \widehat{M}'}$ decomposes as $\widehat{V}_0 \oplus \widehat{V}_1$, with $\widehat{V}_0 $ is spanned by $\hat{a}'=a'+a''$ (the image of $a'$ under $g$) and $\widehat{V}_1'$ by $b'$. 

From the proof of Proposition~\ref{prop:stab}, we see that we have unique spin structures on all eight of the Lagrangian bundles appearing in \eqref{eq:normals}. We also have non-degenerate holomorphic quadratic forms
$$q \in H^0(\Sym^2 V_0^*), \ \ q' \in H^0(\Sym^2 (V'_0)^*), \ \ \hat{q} \in H^0(\Sym^2 \widehat{V}_0^*), \ \ \hat{q}' \in H^0(\Sym^2 (\widehat{V}'_0)^*),$$
all coming from the Killing form on $\g$.

Thus, we obtain stabilization isomorphisms
$$ \Stab: \Pb_{L_0, L_1} \xrightarrow{\phantom{b}\cong\phantom{b}} \Pb_{L'_0, L'_1}, \ \ \Stab': \Pb_{L'_0, L'_1} \xrightarrow{\phantom{b}\cong\phantom{b}} \Pb_{L''_0, L''_1},$$
$$ \hat{\Stab}: \Pb_{L_0, L_1} \xrightarrow{\phantom{b}\cong\phantom{b}} \Pb_{\hat{L}'_0, \hat{L}'_1}, \ \ \hat{\Stab}': \Pb_{\hat{L}'_0, \hat{L}'_1} \xrightarrow{\phantom{b}\cong\phantom{b}} \Pb_{L''_0, L''_1}.$$

Equation~\eqref{eq:Fe} translates into the commutativity of the diagram $$
\xymatrix{
& \Pb_{L'_0, L'_1}  \ar@{^{(}->}^{\Stab'} [rd] & \\
\Pb_{L_0, L_1} \ar@{^{(}->}^{\Stab} [ru]  \ar@{^{(}->}^{\hat{\Stab}} [rd] & & \Pb_{L_0'', L_1''} \\
& \Pb_{\hat{L}'_0, \hat{L}'_1}  \ar@{^{(}->}^{\hat{\Stab}'} [ru] &
}$$

The two compositions $\Stab' \circ \Stab$ and $\hat{\Stab}' \circ \hat{\Stab}$ are both instances of the  maps constructed from Proposition~\ref{prop:stabL}. They are both associated to the inclusion $M \hookrightarrow M''$, and to the same normal bundle decomposition
$$ N_{MM''} = W_0 \oplus W_1,$$
where
$$ W_0 = V_0|_{M} \oplus V_0'|_{M} =  \widehat{V}_0|_{M} \oplus \widehat{V}_0'|_{M} = \Span_{\g}(a', a'')$$
and
$$ W_1 = V_1|_{M} \oplus V_1' |_{M}= \widehat{V}_1|_{M} \oplus \widehat{V}_1'|_{M} = \Span_{\g}(b', b'').$$
There are unique spin structures on $W_0, W_1, \widehat{W}_0,  \widehat{W}_1$. The only difference lies in the quadratic forms on $W_0$ used to apply Proposition~\ref{prop:stabL}. To construct $\Stab' \circ \Stab$, we use the form $q \oplus q'$, whereas to construct $\hat{\Stab}' \circ \hat{\Stab}$, we use $\hat{q} \oplus \hat{q}'$. Concretely, in one case we take the direct sum of the Killing forms on the spans of $a'$ and $a''$, whereas in the other we take the direct sum of the Killing forms on the spans of $a' + a''$ and $a''$.

We now interpolate between these two quadratic forms by taking the direct sum of the Killing forms on the spans of $a' + ta''$ and $a''$, for $t \in [0,1]$. Proposition~\ref{prop:stabL} gives a continuous family of maps 
$$ \Stab_t: \Pb_{L_0, L_1} \to \Pb_{L_0'', L_1''}, \ t \in [0,1]$$
interpolating between $\Stab_0=\Stab' \circ \Stab$ and $\Stab_1 = \hat{\Stab}' \circ \hat{\Stab}$. However, any such family must be constant, because morphisms in the category of perverse sheaves (over $\Z$) are discrete objects. 

We conclude that \eqref{eq:Fe} is satisfied, and therefore handleswap invariance holds.
\end{proof}

Naturality for the objects  $\Pb_{L^{\#}_0, L^{\#}_1} \in \Perv'(R(Y))$ is established in a similar manner, with the additional constraint that we must fix the basepoint $z \in Y$.

This finishes the proofs of Theorem~\ref{thm:main1} and Theorem \ref{thm:framed}.

\section{Properties and examples}
\label{sec:Examples}

\subsection{Dualities}
Our invariants $P^{\bullet}(Y)$ and $P^{\bullet}_\#(Y,z)$ are defined for oriented three-manifolds. However, as the following result shows, they are independent of the orientation on $Y$. 
 
 \begin{proposition}
 Let $Y$ be a closed, connected, oriented three-manifold, and let $-Y$ denote $Y$ with the opposite orientation. Pick a basepoint $z \in Y$. Then, we have isomorphisms
 $$ P^{\bullet}(Y) \xrightarrow{\cong} P^{\bullet}(-Y), \ \ \ P^{\bullet}_\#(Y,z) \xrightarrow{\cong} P^{\bullet}_{\#}(-Y,z).$$
 \end{proposition}
 
 \begin{proof}
 A Heegaard splitting $(\Sigma, U_0, U_1)$ for $Y$ gives a Heegaard splitting for $-Y$, with the orientations on $\Sigma, U_0$ and $U_1$ being reversed. The orientation on $\Sigma$ is involved in the definition of the complex symplectic form $\omega_{\C}$ from~\eqref{eq:omegac}. Reversing the orientation changes the sign of $\omega_{\C}$, but does not affect the complex structure $J$ (since the latter comes from the complex structure on $G=\sl$, not on $\Sigma$). 
 
Let us consider Bussi's construction from Section~\ref{sec:Bussi}. Suppose $(M, \omega)$ is a complex symplectic manifold with an $L_0$-chart $(S, P, U, f, h, i)$. Part of the data is the isomorphism $h: S \to T^*U$. If we denote by $r: T^*U \to T^*U$ the map given by multiplication by $-1$ on the fibers, we find that $(S, P, U, -f, h \circ r, i)$ is an $L_0$-chart for $(M, -\omega)$. Given $f: U \to \C$, note that we can relate $f$ to $-f$ via the family $e^{i\theta}f, \theta \in [0, \pi]$. This gives an isomorphism between the vanishing cycle functors for $f$ and $-f$. (The square of this isomorphism is the monodromy map.) By patching together these isomorphisms, we obtain an isomorphism between the perverse sheaves $\Pb(L_0, L_1)$ defined in $(M, \omega)$ and $(M, -\omega)$. Applying this to our setting, we get the desired claim about the invariants for $Y$ and $-Y$. 
 \end{proof}

We can also ask how $P^{\bullet}(Y)$ and $P^{\bullet}_\#(Y, z)$ behave under Verdier duality. In \cite[Theorem 2.1]{Bussi}, Bussi shows that, for any spin complex Lagrangians $L_0$ and $L_1$, the perverse sheaf $\Pb_{L_0, L_1}$ is naturally isomorphic to its Verdier dual. As a consequence, we have
\begin{proposition}
\label{prop:Vsd}
The invariants $P^{\bullet}(Y) \in \Perv'(\CharIrr(Y))$ and $ P^{\bullet}_\#(Y,z) \in \Perv'(\Rep(Y))$ are Verdier self-dual.
\end{proposition}

Starting from $P^{\bullet}(Y)$ and $P^{\bullet}_\#(Y,z)$, we defined $\HP^*(Y)$ and $\HPf^*(Y,z)$ by taking hypercohomology. We could alternatively take hypercohomology with compact support, and define
$$ \HPcs^*(Y) := \HH_{\cs}^*( P^{\bullet}(Y)),$$
$$ \HPfcs^*(Y,z) := \HH_{\cs}^*( P^{\bullet}_{\#}(Y,z)).$$
From \eqref{eq:Vdual} and Proposition~\ref{prop:Vsd} we obtain duality isomorphisms 
$$ \HPcs^k(Y) \cong \Hom(\HP^{-k}(Y), \Z) \oplus \Ext^1(\HP^{-k-1}(Y), \Z)$$
and
$$ \HPfcs^k(Y,z) \cong \Hom(\HPf^{-k}(Y,z), \Z) \oplus \Ext^1(\HPf^{-k-1}(Y,z), \Z).$$

Observe also that, since we use sheaf cohomology, the invariants $\HP^*(Y)$ and $\HPf^*(Y,z)$ are models for Floer cohomology, rather than homology. We can define homological invariants $\HP_*(Y)$ and $\HP^{\#}_{*}(Y,z)$ by dualizing the complexes that define $\HP^*(Y)$ resp. $\HPf^*(Y,z)$, and then taking homology. We have 
$$\HP_k(Y) \cong \Hom(\HP^{k}(Y), \Z) \oplus \Ext^1(\HP^{k+1}(Y), \Z) \cong \HPcs^{-k}(Y)$$
and
$$\HP^{\#}_k(Y,z) \cong \Hom(\HPf^{k}(Y,z), \Z) \oplus \Ext^1(\HPf^{k+1}(Y,z), \Z) \cong \HPfcs^{-k}(Y,z).$$

\subsection{Computational tools}
To calculate the perverse sheaf invariants $P^{\bullet}(Y)$ and $P^{\bullet}_\#(Y,z)$ in specific examples, we will rely on Theorem~\ref{thm:main2} from the Introduction.

\begin{proof}[Proof of Theorem~\ref{thm:main2}]
By Lemmas~\ref{lem:cleanChar} and \ref{lem:cleanRep}, regularity of the underlying schemes is equivalent to the condition that the Lagrangians intersect cleanly. The desired result now follows from Proposition~\ref{prop:BClean}.
\end{proof}

The following lemma describes a simple situation where we can identify the local system in Theorem~\ref{thm:main2}(b). 

\begin{lemma}
\label{lem:trivial}
Let $Y$ be a closed, connected, oriented three-manifold, $z \in Y$ a basepoint, and $\rho \in \Rep(Y)$ a reduced, irreducible representation. We assume that $[\rho]$ is isolated in the character variety $\Char(Y)$, so that (by Lemma~\ref{lem:cleanRep}) the Lagrangians $L_0^{\#}$ and $L_1^{\#}$ intersect cleanly along the orbit $Q:=\O_{\rho}$. Then, we have an isomorphism 
$$\P^{\bullet}_{\#}(Y, z)|_Q \cong \Z_Q[3].$$
\end{lemma}

\begin{proof}
Observe that $Q$ is diffeomorphic to $\Gad = \psl \cong \rp^3 \times \R^3$. From Theorem~\ref{thm:main2}(b) we know that $\P^{\bullet}_{\#}(Y)|_Q$ is a local system over $Q$ with fiber $\Z$, in degree $-3$. Since $H^1(Q; \Z_2) \cong \Z_2$, there are two possibilities for the local system. To show that it is the trivial one, we will use Lemma~\ref{lem:vanish-cycl-desc-1}.

Recall from the proof of Lemma~\ref{lem:cleanRep} that we have an inclusion $\Rep(\Sigma) \hookrightarrow  M^{\#}$, and the Lagrangians $L^{\#}_0$ and $L^{\#}_1$ live inside $\Rep(\Sigma)$. In the situation at hand, at any point $x \in Q$, by the clean intersection condition we have that $T_x L^{\#}_0 \cap T_xL^{\#}_1=T_xQ$ is three-dimensional. Therefore, we must have
$$T_x L^{\#}_0 + T_xL^{\#}_1= T_x (\Rep(\Sigma)) \subset T_xM^{\#}.$$
We deduce that the symplectic normal bundle of $Q$ is
$$NQ = T_x (\Rep(\Sigma))/T_xQ.$$
  
The group $\Gad$ acts transitively on $Q$. There is no natural action of $\Gad$ on the ambient manifold $M^{\#} = \Xtw(\Sigma^{\#})$, but there is one (given by conjugation) on the subvariety $\Rep(\Sigma)$, and this action preserves the Lagrangians $L^{\#}_0$ and $L^{\#}_1$. Hence, we get a $\Gad$-action on the normal bundle $NQ$, which preserves the decomposition
$$  NQ \cong N_0 Q \oplus N^*_0 Q.$$
considered in \eqref{eq:linear_polarization_normal_intersection}. The $\Gad$-action gives a trivialization of the bundles $N_0Q$ and $N^*_0Q$ over $Q$. By choosing a polarization of $NQ$ transverse to $N_0Q$ and $N^*_0Q$ at some $x \in Q$, we can use the $\Gad$-action to extend it to such a polarization at all points of $Q$. For this polarization, the bundle $W^+$ defined in Section~\ref{sec:Clean} is clearly trivial.

In view of Lemma~\ref{lem:vanish-cycl-desc-1}, the only thing that remains to be proved is that the isomorphism $TL^{\#}_1|_Q \to TL^{\#}_0|_Q$ from \eqref{eq:TLL} preserves spin structures. To do this, recall that the spin structures on $L^{\#}_0=\Rep(U_0)$ and $L^{\#}_1=\Rep(U_1)$ are unique (because the Lagrangians are simply connected). The same is true for Lagrangians  $L_0=\CharIrr(U_0), L_1=\CharIrr(U_1) \subset \CharIrr(\Sigma)$, which intersect transversely at the point $[\rho]$. Furthermore, if we consider the open subsets 
$$\tL_i := \RepIrr(U_i) \subset L^{\#}_i = \Rep(U_i),i=0,1,$$
we can see from the proof of Lemma~\ref{lem:pifree} that these are also simply connected. There are  natural projections $p_i : \tL_i \to L_i$, with fibers $\Gad$, and therefore we have isomorphisms
$$ T\tL_i \cong p_i^*TL_i \oplus \g.$$
By the uniqueness of the spin structures on $\tL_i$ and $L_i$, we can think of the spin structure on $T\tL_i$ as obtained from the one on $L_i$ via pull-back and adding the trivial spin structure on $\g$. 

When restricted to $Q$, we can also identify the pull-backs $p_i^*TL_i$ with the normal bundles $N_iQ$. After these identifications, the projection 
\begin{equation}
\label{eq:TLQ}
T\tL_1|_Q \to T\tL_0|_Q
\end{equation} is the direct sum of the identity on $TQ \cong \g$ and a $\Gad$-equivariant projection $N_1Q \to N_0Q$. This second summand is the pull-back of a projection $T_{[\rho]}L_1 \to T_{[\rho]}L_0$, which must preserve spin structures. (A spin structure on a vector bundle over a point, i.e. on a vector space, is unique.) Note that the spin structures on $T\tilde L_i|_Q$ are equivariant under the $\Gad$-action, because they are restrictions of the spin structures on the whole of $\tilde L_i$, which are unique and therefore obtained by pull-back from the ones on $L_i$. Once we have this, we see that the $\Gad$-equivariant isomorphism \eqref{eq:TLQ} matches the spin structures on $\tL_1$ and $\tL_0$. By uniqueness, these are exactly the restrictions of the spin structures on $L^{\#}_0$ and $L^{\#}_1$.
\end{proof}

\subsection{Examples}
We present a few calculations, for some of the examples discussed in Section~\ref{sec:ex3}. We only look at situations where the underlying scheme is regular, so that we can apply Theorem~\ref{thm:main2}. In these cases, the perverse sheaf under consideration is a local system with fibers $\Z$, supported in degrees $-k$, where $k$ is the complex dimension of the respective component of $\CharIrr(Y)$ or $\Rep(Y)$. We will use the subscript $(i)$ to denote a group in degree $i$.

For $Y=S^3$, we have $\CharIrr(S^3) =\emptyset$ and $\Rep(S^3)$ is a point, so
$$\HP^*(S^3)=0, \ \ \ \HPf^*(S^3) = \Z_{(0)}.$$

For $Y$ being the connected sum of $k$ copies of $S^1 \times S^2$ (cf. Example~\ref{ex:free} and Section~\ref{sec:free}), the sheaf $\P^{\bullet}_{\#}(Y)$ is a local system with fibers $\Z$ (in degree $-3k$) over $G^k$. Since $G \cong S^3 \times \R^3$ is simply connected, the local system must be trivial, and we get
$$ \HPf^*(\#^k (S^1 \times S^2)) \cong \Z^{k}_{(-3)} \oplus \Z^{k}_{(0)}.$$
When $k=1$, there are no irreducible representations and therefore 
$$ \HP^*(S^1 \times S^2)=0.$$
For $k=2$, the space $\CharIrr(F_2)$ is not simply connected (see Remark~\ref{rem:doi}), and it is not immediately clear how to identify the local system $\P^{\bullet}(Y)$. However, for all $k \geq 3$, we have $\pi_1(\CharIrr(F_k))=1$ by Lemma~\ref{lem:pifree}, and therefore 
$$\HP^*(\#^k (S^1 \times S^2)) \cong H^{*+3k-3}(\CharIrr(F_k); \Z).$$

Next, we will look at lens spaces $L(p,q)$ and Brieskorn spheres $\Sigma(p,q,r)$. For these manifolds, the computations of $\HP^*$ and $\HPf^*$ were stated in the Introduction, in Theorems~\ref{thm:lens} and \ref{thm:Brieskorn}.

\begin{proof}[Proof of Theorem~\ref{thm:lens}]
Lens spaces were discussed in Example~\ref{ex:Lpq}. Since $\pi_1$ is Abelian, there are no irreducible representations, and $\HP^*(L(p,q))=0$. To calculate $\HPf^*(L(p,q))$, note that $R(Y)$ is the disjoint union of some points and copies of $TS^2$. Over the points, the perverse sheaf $P^{\bullet}_\#(Y)$ is a copy of $\Z$ in degree $0$, and over each copy of $TS^2$, it is a local system with fibers $\Z$ in degree $-2$. Since $TS^2$ is simply connected, the local system is trivial. After taking cohomology, we get the advertised answer.
\end{proof}

\begin{proof}[Proof of Theorem~\ref{thm:Brieskorn}]
The Brieskorn spheres $\Sigma(p,q,r)$ were considered in Example~\ref{ex:Br3}. The variety $\CharIrr(\Sigma(p,q,r))$ consists of $N=(p-1)(q-1)(r-1)/4$ isolated points, so $\HP^*(\Sigma(p,q,r))$ is $\Z^N$ in degree zero.

To compute $\HPf^*$, recall that the representation variety is composed of a point and $N$ copies of $\psl \cong \rp^3 \times \R^3$. The perverse sheaf $P^{\bullet}_\#(Y)$ is $\Z$ over the point, and (by Lemma~\ref{lem:trivial}) the trivial local system with fiber $\Z$ in degree $-3$ over each copy of $\psl$. This gives the desired answer.
\end{proof}

Lastly, we consider $\HP^*$ for the Seifert fibered homology spheres $\Sigma(a_1, \dots, a_n)$ discussed in Example~\ref{ex:Br4}. Then, the variety $\CharIrr(Y)$ is the disjoint union of simply connected components $\M_{\alpha}$, of dimensions $2m(\alpha)-6$. It follows that 
\begin{equation}
\label{eq:SF}
\HP^*( \Sigma(a_1, \dots, a_n)) \cong \bigoplus_{\alpha} H^{*+2m(\alpha)-6}(\M_{\alpha}; \Z).
\end{equation}
The Poincar\'e polynomials of $\M_{\alpha}$ were computed in \cite{BodenYokogawa}.

\subsection{The Euler characteristic}
As noted in the Introduction, the Euler characteristic of Floer's $\su$ instanton homology is twice the Casson invariant; cf. \cite{TaubesCasson}. The Euler characteristic of the framed theory $I^\#(Y)$ is less interesting, being equal to the order of $H_1(Y)$ if $b_1(Y)=0$, and zero otherwise; cf. \cite{Scaduto}. 

In our context, we define the (sheaf-theoretic) {\em full $\sl$ Casson invariant} of $Y$ to be the Euler characteristic of $\HP^*(Y)$:
\begin{equation}
\label{eq:lambdaP}
 \lambda^P(Y) := \sum_{k \in \Z} (-1)^k \cdot \rk \HP^k(Y).
\end{equation}

The following proposition shows that the right hand side of \eqref{eq:lambdaP} is well-defined. 

\begin{proposition}
For any closed, oriented $3$-manfiold $Y$, the invariants $\HP^*(Y)$ and $\HPf^*(Y)$ are finitely generated as Abelian groups.
\end{proposition}

\begin{proof}
 By \cite[Theorem 3.1]{Bussi}, the intersection of complex Lagrangians is an (oriented) complex analytic d-critical locus. The perverse sheaf $\Pb(Y) = \Pb_{L_0, L_1}$ is isomorphic to the one constructed in \cite[Theorem 6.9]{BBDJS}. The manifold $M=\CharIrr(\Sigma)$ is also an algebraic variety, and the Lagrangians $L_0, L_1$ are algebraic. Thus, $L_0 \cap L_1$ is naturally an algebraic d-critical locus, and from this we get an algebraic perverse sheaf $\Pb_{\operatorname{alg}}(Y)$. By construction, $\Pb_{\operatorname{alg}}(Y)$ is taken to $\Pb(Y)$ by the forgetful functor from algebraic to complex analytic perverse sheaves. This implies that the cohomology sheaves of $\Pb(Y)$ are constructible for an algebraic stratification of $\CharIrr(Y)= L_0 \cap L_1$, which must have finitely many strata. We conclude that $\HP^*(Y)$ is finitely generated. A similar argument applies to $\HPf^*(Y)$.
\end{proof}

The invariant $\lambda^P$ should be contrasted with the $\sl$ Casson invariant of three-manifolds defined by Curtis in \cite{Curtis}, which we will denote by $\lambda^C$. Her invariant counts only isolated irreducible representations. 

For example, for the Brieskorn spheres $\Sigma(p,q,r)$, all the irreducible representations are isolated, and we have
$$\lambda^P(\Sigma(p,q,r)) = \lambda^C(\Sigma(p,q,r)) = (p-1)(q-1)(r-1)/4.$$

On the other hand, for the more general Seifert fibered homology spheres $\Sigma(a_1, \dots, a_n)$, by \cite[Theorem 2.7]{BodenCurtis}, we have
$$\lambda^C(\Sigma(a_1, \dots, a_n)) = \sum_{1 \leq i_1 < i_2 < i_3 \leq n} \frac{(a_{i_1} - 1)(a_{i_2}-1)(a_{i_3}-1)}{4}.$$

To calculate $\lambda^P(\Sigma(a_1, \dots, a_n))$, we use \eqref{eq:SF} and the fact that the Euler characteristic of the spaces $\M_{\alpha}$ is $(m(\alpha)-1)(m(\alpha)-2)2^{m(\alpha)-4}$; cf. \cite{BodenYokogawa}. We obtain
$$\lambda^P(\Sigma(a_1, \dots, a_n)) = \sum_{\alpha} (m(\alpha)-1)(m(\alpha)-2)2^{m(\alpha)-4},$$

For a concrete example, take the homology sphere $\Sigma(2,3,5,7)$. This has $23$ isolated irreducible representations, and six (complex two-dimensional) families of irreducibles with $m(\alpha)=4$.  Therefore,
$$ \lambda^C(\Sigma(2,3,5,7)) = 23 \ \ \text{but} \ \
\lambda^P(\Sigma(2,3,5,7)) = 23 + 6 \cdot 6 = 59.$$

\subsection{A bound on degrees} We now prove another result from in the Introduction.

\begin{proof}[Proof of Theorem~\ref{thm:ikshu}]
Note that $\Rep(Y)$ and $\Char(Y)$ are affine varieties, and $\CharIrr(Y) \subset \Char(Y)$ an open subvariety. In general, the hypercohomology of any perverse sheaf on a complex algebraic variety of dimension $d$ vanishes in degrees outside the interval $[-d, d]$; see for example \cite[Proposition 5.2.20]{Dimca}. Furthermore, as a consequence of the Artin vanishing theorem, if the underlying variety is affine, then the hypercohomology of a perverse sheaf is supported in non-positive degrees; see \cite[Corollary 5.2.18]{Dimca}.

If $Y$ has a Heegaard splitting of genus $g$, then the Lagrangians $L_i$ are isomorphic to $\CharIrr(F_g)$ and hence have complex dimension $3g-3$. The dimension of $\CharIrr(Y)$ is bounded above by this. This shows that $\HP^*(Y)$ is supported in degrees in the interval $[-3g+3, 3g-3]$. 

Similarly,  the dimension of $\Rep(Y)$ is bounded above by $3g$. Since $\Rep(Y)$ is affine, it follows that $\HPf^*(Y)$ is supported in degrees in $[-3g, 0]$.
\end{proof}

\section{Further directions}
\label{sec:Further}

\subsection{Other groups}
\label{sec:otherG}
The sheaf-theoretic Floer cohomologies defined in this paper were based on the Lie group $\sl$. One may ask about generalizations to other complex reductive Lie groups $G$. 

We refer to \cite{Sikora} for a discussion of $G$-representations of $\Gamma = \pi_1(M)$, where $M$ is either a surface or a $3$-manifold with boundary (such as a handlebody). Let us review a few  definitions and facts.

A representation $\rho: \Gamma \to G$ is called {\em irreducible} if $\rho(\Gamma)$ is not contained in any proper parabolic subgroup of $G$. Further, an irreducible representation $\rho$ is called {\em good} if the stabilizer of its image is the center of $G$. The categorical quotient $\Char_G(M) = \Hom(\pi_1(M), G)  \sslash \Gad$ is called the $G$-character variety. It has open subsets
$$  \Char_{G, \operatorname{good}}(M) \subset \Char_{G, \irr}(M) \subset \Char_G(M)$$
corresponding to the good, resp. irreducible representations. 

We will focus our attention on complex semisimple Lie groups $G$. For such groups, the Killing form on their Lie algebra $\g$ is non-degenerate. The existence of a symmetric, bilinear, invariant form on $\g$ is an ingredient in both Goldman's results on the symplectic structure nature of the fundamental group of surfaces \cite{Goldman}, and in our proof of stabilization invariance (where it gives the form $q$ needed in Proposition~\ref{prop:stabL}). 

Let $G$ be a complex semisimple Lie group, and $\Sigma$ a closed orientable surface of genus $g \geq 2$. Then, $\Char_{G, \irr}(\Sigma)$ is an orbifold, and its open subset $\Char_{G, \operatorname{good}}(\Sigma)$ is a smooth manifold. (See \cite[Proposition 5]{Sikora}.) Moreover, Goldman \cite{Goldman} showed that $\Char_{G, \operatorname{good}}(\Sigma)$ can be equipped with a holomorphic symplectic form. If we have a Heegaard decomposition $Y^3 = U_0 \cup_{\Sigma} U_1$, then the image of $\Char_G(U_i)$ in $\Char_G(\Sigma)$ intersects $\Char_{G, \operatorname{good}}(\Sigma)$ in a complex Lagrangian submanifold; cf. \cite[Theorem 6]{Sikora}. 

When $G = \sln$, we have the further nice property that all irreducible representations are good. Thus, $\Char_{G, \irr}(\Sigma)$ is a complex symplectic manifold, with Lagrangians coming from the Heegaard decomposition of $Y^3$. By applying Bussi's construction we obtain a perverse sheaf $P^{\bullet}(Y, G)$ over $\Char_{G, \irr}(Y)$. The same proof as in the $\sl$ case carries over to $\sln$, and we get that $P^{\bullet}(Y, G)$ is a natural invariant of $Y$. Its hypercohomology
$$  \HP^*(Y, G) := \HH^*( P^{\bullet}(Y, G))$$
is called the {\em sheaf-theoretic $\sln$ Floer cohomology of $Y$}. 

For other complex semisimple Lie groups, we could restrict to the open set consisting of good representations, and proceed as before. This is somewhat unnatural, but gives rise to invariants. A more challenging project would be to work on the orbifold $\Char_{G, \irr}(\Sigma)$, and produce invariants that take into account all irreducible flat connections. Of particular interest is the case $G = \psl$, which is the most relevant one for Witten's interpretation of Khovanov homology (cf. Section~\ref{sec:khovanov} below). We remark that in \cite{CurtisPSL}, Curtis defined a $\psl$ Casson invariant for three-manifolds; her invariant is a count of the isolated irreducible flat connections, with rational weights dictated by the orbifold structure.

With regard to constructing framed (sheaf-theoretic) Floer cohomologies, for $G=\sln$ we can draw inspiration from the constructions of $U(n)$ Floer homologies in \cite{KMknots} and \cite{WWFloerField}. Specifically, for $\Sigma$ and $\Sigma^{\#} = \Sigma \# T^2$ with a basepoint $w \in T^2$ encircled by a curve $\gamma$ as before, and for any integer $d$ relatively prime to $n$, we consider a twisted character variety
$$ \Char_{n, d, \tw}(\Sigma^{\#}) = \{\rho: \pi_1(\Sigma^{\#} \setminus \{w\}) \to G \mid \rho(\gamma)= \exp({2\pi i d/n} ) \cdot I \} / \Gad.$$
This is a complex symplectic manifold, and a Heegaard decomposition of $Y$ along $\Sigma$ produces two Lagrangians inside $\Char_{n, d, \tw}(\Sigma^{\#})$, just as in Section~\ref{sec:Lags}. We are using here that $\Char_{n, d, \tw}(T^2)$ is a point. We get that the intersection of the two Lagrangians can be identified with the representation variety of $Y$, and Bussi's construction gives a perverse sheaf $P^{\bullet}_{\#}(Y,z)$ on that variety. Invariance can be proved as in Section~\ref{sec:Invariant}.

\subsection{Extensions}
\label{sec:extensions}
Going back to the case $G=\sl$, there are a number of ways one could try to extend the constructions in this paper:
\begin{itemize}
\item There should be versions of the sheaf-theoretic Floer cohomology for admissible $\gl$ bundles, and for knots and links in three-manifolds;
\item
There should be a $\psl$-equivariant sheaf-theoretic Floer cohomology of three-manifolds, which involves both the reducibles and the irreducibles; 
\item
An alternate construction of three-manifold invariants should be given using derived algebraic geometry, cf. Remark~\ref{rem:PTVV}; 
\item
Similar invariants to those in this paper could be constructed using the theory of deformation quantization modules;
\item
We expect our invariants to be functorial under four-dimensional cobordisms, and thus part of  $3+1$ dimensional TQFTs, based on the Kapustin-Witten or Vafa-Witten equations; 
\item
We expect that $\HP^*$ can be categorified to give an $A_{\infty}$-category associated to the three-manifold, in the spirit of  \cite{KapustinRozansky}, \cite{Haydys}, or \cite{GMWfans};
\item
One can investigate the effect on $\HP^*$ or $\HP^*_{\#}$ induced by varying the complex structure on the moduli space of flat connections.
\end{itemize}

\subsection{Relation to Khovanov homology}
\label{sec:khovanov}
In \cite{Khovanov}, Khovanov defined a homology theory for knots and links in $\R^3$, now known as Khovanov homology.  Witten \cite{FiveBranes} conjectured that the Khovanov homology of a link $L \subset S^3$ can be understood as a version of Floer homology, using the Haydys-Witten equation on $\R^3 \times \R_+ \times \R$, with certain boundary conditions. The generators of this Floer complex are solutions to the Kapustin-Witten equations \cite{KapustinWitten} on $\R^3 \times \R_+$.

Extending Khovanov homology to an invariant of links in arbitrary three-manifolds is an open problem. It is natural to attempt to do so by considering the Haydys-Witten equations on $Y \times \R_+ \times \R$, where $Y$ is any three-manifold. There are formidable analytical difficulties to be overcome in order to carry out this program, having to do with non-compactness of the moduli spaces; see \cite{Taubes3, Taubes4, Taubes5}. We refer to \cite{GPV}, \cite{GPPV} for some expectations about the resulting invariants, coming from the physics perspective. 

The sheaf-theoretic invariant $\HP^*(Y)$ constructed in this paper is a small step in this program. It is meant to give $\sl$ Floer homology, which can be thought of as encoding information from the Kapustin-Witten equations on $Y \times \R$. We can view $\sl$ Floer homology as the space of integration cycles (thimbles) for the complex Chern-Simons functional, as in \cite{WittenAC}, \cite{FiveBranes}, \cite{WittenLectures1}. To obtain analogues of the Jones polynomial, one would need to also introduce the boundary conditions at $Y \times \{0\}$. Moreover, to get to analogues of Khovanov homology, one would then need to categorify these  invariants.

\subsection{An open question} Zentner \cite{Zentner} proved that if $Y$ is a non-trivial integral homology $3$-sphere, then $\pi_1(Y)$ admits an irreducible representation into $\sl$.

\begin{question}
Can one use Zentner's result to prove that $\HP^*(Y)$ detects $S^3$ among homology spheres? 
\end{question}

\bibliographystyle{custom}
\bibliography{biblio}

\end{document}